\documentclass[
	english
]{siamart1116}

\usepackage[T1]{fontenc}
\usepackage[utf8]{inputenc}
\usepackage{microtype} 

\usepackage{amsmath,amsfonts,amssymb}

\usepackage{enumitem}
\setenumerate{label=(\roman*)}

\usepackage{dsfont}

\usepackage{mathtools} 

\usepackage{lmodern}

\newcommand{\mscLink}[1]{\href{http://www.ams.org/mathscinet/msc/msc2010.html?t=#1}{#1}}

\makeatletter
\def\@delim#1#2#3#4#5{\mathopen{#1#2} {#3} \mathclose{#4#5}}
\def\@provide@middle#1{\def\mid{#1|}}
\def\varh#1#2#3{\mathopen{}\mathclose\bgroup\left#1 \def\mid{\mathrel{}\mathclose{}\middle|\mathopen{}\mathrel{}}#3 \aftergroup\egroup\right#2}
\def\normalh#1#2#3{\@delim{}{#1}{\def\mid{\mathrel{|}} #3}{}{#2}}
\def\bigh#1#2#3{\@delim{\bigl}{#1}{\@provide@middle{\bigm}#3}{\bigr}{#2}}
\def\Bigh#1#2#3{\@delim{\Bigl}{#1}{\@provide@middle{\Bigm}#3}{\Bigr}{#2}}
\def\biggh#1#2#3{\@delim{\biggl}{#1}{\@provide@middle{\biggm}#3}{\biggr}{#2}}
\def\Biggh#1#2#3{\@delim{\Biggl}{#1}{\@provide@middle{\Biggm}#3}{\Biggr}{#2}}
\def\Bigggh#1#2#3{\@delim{\bBigg@{3}}{#1}{\@provide@middle{\mathrel\bBigg@{3}}#3}{\bBigg@{3}}{#2}}
\def\Biggggh#1#2#3{\@delim{\bBigg@{4}}{#1}{\@provide@middle{\mathrel\bBigg@{4}}#3}{\bBigg@{4}}{#2}}
\makeatother

\newcommand\norm[1]{\lVert#1\rVert}
\newcommand\abs[1]{\lvert#1\rvert}
\newcommand\bigabs[1]{\bigl\lvert#1\bigr\rvert}

\newcommand\dual[2]{\langle #1, #2\rangle}
\newcommand\bigdual[2]{\bigl\langle #1, #2\bigr\rangle}

\newcommand\N{\mathbb{N}}
\newcommand\R{\mathbb{R}}

\renewcommand\d{\mathrm{d}}

\newcommand{\weakly}{\rightharpoonup}
\newcommand{\weaklystar}{\stackrel\star\rightharpoonup}

\newcommand{\anni}{^\perp}

\newcommand{\adjoint}{^\star}

\newcommand\HH{\mathcal{H}}
\newcommand\LL{\mathcal{L}}
\newcommand\TT{\mathcal{T}}
\newcommand\RR{\mathcal{R}}
\newcommand\KK{\mathcal{K}}
\newcommand\MM{\mathcal{M}}
\newcommand\NN{\mathcal{N}}
\newcommand\ZZ{\mathcal{Z}}

\newcommand\TTs{\mathcal{T}^{\star}}
\newcommand\NNs{\mathcal{N}^{\star}}
\newcommand\KKs{\mathcal{K}^{\star}}

\DeclareMathAlphabet{\mathpzc}{OT1}{pzc}{m}{it}
\newcommand\oo{\mathpzc{o}}
\newcommand\OO{\mathcal{O}}

\newcommand\I{\mathds{1}}

\newcommand\dist{\operatorname{dist}}

\newcommand\sign{\operatorname{sign}}

\let\subseteq\subset
\let\supseteq\supset

\newtheorem{assumption}[theorem]{Assumption}
\newtheorem{remark}[theorem]{Remark}
\newtheorem{example}[theorem]{Example}

\numberwithin{theorem}{section}

\begin{document}
\title{No-Gap Second-Order Conditions via a Directional Curvature Functional%
\footnote{\tiny
This research was supported by the German Research Foundation under grant numbers ME 3281/7-1 and WA 3636/4-1
 within the priority program 
``Non-smooth and Complementarity-based Distributed Parameter Systems: Simulation and Hierarchical Optimization'' 
 (SPP 1962).}
}

\author{%
	Constantin Christof%
	\footnote{\tiny%
		Technische Universität München,
		Faculty of Mathematics,
		85748 Garching,
		Germany
		}%
	\and
	Gerd Wachsmuth%
	\footnote{\tiny%
		Brandenburgische Technische Universität Cottbus-Senftenberg,
		Institute of Mathematics,
		Chair of Optimal Control,
		03046 Cottbus,
		Germany,
		\url{https://www.b-tu.de/fg-optimale-steuerung},
		\email{gerd.wachsmuth@b-tu.de}%
	}
	\orcid{0000-0002-3098-1503}%
}
\maketitle

\begin{abstract}
	This paper is concerned with necessary and sufficient
	second-order conditions for finite-dimensional and infinite-dimensional constrained
	optimization problems. Using a suitably defined directional curvature
	functional for the admissible set, we derive no-gap second-order optimality conditions in an abstract functional analytic setting.
	Our theory not only covers those cases where the classical assumptions of
	polyhedricity  or second-order regularity are satisfied but also allows to
	study problems in the absence of these requirements. As a tangible example, we
	consider no-gap second-order conditions for bang-bang optimal control
	problems.
\end{abstract}



\section{Introduction}
\label{sec:intro}

The aim of this paper is to develop a theoretical framework for necessary and
sufficient second-order optimality conditions (henceforth referred to as SNC and SSC, respectively).
The main feature of our approach is that we use a suitably defined ``directional'' curvature functional to take into account the influence of the admissible set. 
Our definition of curvature allows us to derive no-gap second-order conditions that provide more flexibility than classical results. 
In particular, our approach makes it possible to exploit
additional information about the gradient of the objective functional. Such
information is, e.g., often available in the optimal control of partial
differential equations (PDEs) where the gradient of the objective is typically
characterized by an adjoint equation and, as a consequence, enjoys additional
regularity properties.

	Let us briefly clarify what we understand by ``no-gap second-order conditions''.
	For simplicity, we focus on the minimization of a smooth function $f$ over $\R^d$.
	In this case, it is well known that local optimality of $\bar x$ implies
	$\nabla f(\bar x) = 0$ and $h^\top \nabla^2 f(\bar x) \, h \ge 0$
	for all $h \in \R^d$.
	On the other hand,
	$\nabla f(\bar x) = 0$ and $h^\top \nabla^2 f(\bar x) \, h > 0$
	for all $h \in \R^d \setminus \{0\}$
	is equivalent to $\bar x$ being a local minimizer satisfying a quadratic growth condition.
	Hence, the only difference between the necessary and the sufficient condition
	is a non-strict vs.\ a strict inequality in the second-order condition.
	Moreover, this change is as small as possible.
	Such a pair of optimality conditions
	is denoted as ``no-gap second-order conditions''.

The analysis found in this paper originated from the
idea to extend the results of  \cite{CasasWachsmuthsBangBang}.
 In this paper,
the authors derived an SSC for a class of bang-bang
optimal control problems that does not fit into the classical setting of
polyhedricity and second-order regularity,
cf.\ 
\cite{BonnansShapiro}.
Our results turned out to be of relevance for other problems
as well and ultimately gave rise to the abstract framework of
\Cref{sec:second-order_conditions} that not only covers  large parts of the
classical SNC and SSC theory but also allows to study situations where the
admissible set exhibits a singular or degenerate curvature behavior,
cf.\ \Cref{sec:advantages}.
We hope that with the subsequent  analysis we can, on
the one hand, offer an alternative view on well-known SNC and SSC results and,
on the other hand, also provide some new ideas for the study of problems that
do not satisfy the classical assumptions of polyhedricity and second-order
regularity.

Let us give some references:
There are several contributions addressing
second-order optimality conditions for optimization problems posed in infinite-dimensional spaces.
We mention exemplarily
\cite{MaurerZowe1979,BonnansCominettiShapiro1999,BonnansZidani1999,BonnansShapiro,CasasTroeltzsch2012,Casas2015}.
No-gap optimality conditions in
the infinite-dimensional case can be found, e.g., in
\cite[Theorems~2.7, 2.10]{BonnansZidani1999},
\cite[Theorems~2.2, 2.3]{CasasTroeltzsch2012}, 
and \cite[Theorem 5.7]{WachsmuthPolyhedricity}. 
Note that the latter results all rely on the concept of polyhedricity (in contrast to our \Cref{thm:no-gap-SOC})
and require that the Hessian of the Lagrangian is a Legendre form.
In the finite-dimensional case,
one can further employ the notion of second-order regularity to derive
no-gap optimality conditions, see, e.g., \cite{BonnansCominettiShapiro1999}.
Note that, in our approach, the Legendre form condition is substituted
by a more general non-degeneracy condition,
see \eqref{eq:NDS} in \Cref{thm:SSC_wo}
and the discussion in \Cref{subsec:NDS}.

Before we begin with our analysis, we give a short overview over the contents and the structure of this paper:

In \Cref{sec:notation}, we clarify the notation, make our assumptions precise and recall several definitions that are needed for our investigation.

In  \Cref{sec:curvature_functional}, we define the directional curvature functional that is at the heart of our SNC and SSC analysis.
Here, we also discuss basic properties of the curvature functional as, e.g., positivity and lower semicontinuity, that are used in the remainder of the paper.

\Cref{sec:second-order_conditions} addresses SNC and SSC for constrained optimization problems on an abstract functional analytic level.
The main results of this section (and of the paper as a whole) are  \Cref{thm:SNC,thm:no-gap-SOC,thm:SSC_wo}.
These theorems illustrate the advantages of working with the directional curvature functional  and demonstrate that our approach allows for a very short and elegant derivation of the second-order theory.

In \Cref{sec:5}, we demonstrate that the framework of
\Cref{sec:second-order_conditions} indeed covers the classical SNC and
SSC theory for optimization problems with polyhedric or second-order regular
sets.  Here, we further comment on how the assumptions \eqref{eq:NDS} and
\eqref{eq:MRC} appearing in our second-order conditions can be verified in
practice and interpreted in the context of generalized Legendre forms and
Tikhonov regularization.  \Cref{sec:5} also includes two tangible examples
that demonstrate the usefulness of the theorems in
\Cref{sec:second-order_conditions}.

\Cref{sec:advantages} is devoted to problems that are covered by our analysis
but do not fall under the scope of the classical SNC- and SSC-framework. The
first example that we consider in this context is a finite-dimensional optimization
problem whose admissible set exhibits a singular curvature behavior. In
\Cref{subsec:bangbang}, we then study no-gap second-order conditions for
bang-bang optimal control problems in the measure space $\MM(\Omega)$. 
Here, we prove a novel SNC for bang-bang controls and further sharpen the 
SSC in \cite{CasasWachsmuthsBangBang} to close the gap between the two conditions, 
see \Cref{thm:explicitno-gap-measures,thm:no-gap-measures},
and the comparison in \Cref{ex:semilinear_bang_bang}.

\Cref{sec:conclusion} summarizes our findings and gives some pointers to further research. 

\section{Notation, Preliminaries and Basic Concepts}
\label{sec:notation}

Throughout this paper, we always consider the following situation.

\begin{assumption}[Standing Assumptions and Notation]~
	\label{asm:standing_assumption}
	\begin{enumerate}
		\item
			$X$ is the (topological) dual of a separable Banach space $Y$,
		\item
			$\iota : Y \to X^\star$ denotes the canonical embedding of $Y$ into the dual $X^\star$ of $X$,
		\item
			the admissible set $C$ is a closed, non-empty subset of $X$.
	\end{enumerate}
\end{assumption}

Note that, under the above assumptions, $X$ is necessarily a Banach space with a \mbox{weak-$\star$} sequentially compact unit ball, cf.\ the Banach–Alaoglu theorem. We remark that the overwhelming majority of our results also holds when the space $Y$ is assumed to be reflexive instead of separable. We restrict our analysis to the above setting since it allows to study more interesting practical examples, see \Cref{subsec:bangbang}.
For our analysis, we need the following classical concepts (cf.\ \cite[Section 2.2.4]{BonnansShapiro}).

\begin{definition}
	\label{eq:def_cones}
	Let $x \in C$ be given. We define
	the radial cone, the strong outer tangent cone and  the \mbox{weak-$\star$} outer tangent cone 
	to $C$ at $x$, respectively, 
	by
	\begin{equation*}
	\begin{aligned}
		\RR_C(x) &:=
		\varh\{\}{h \in X \mid  \exists T > 0 \ \forall t \in [0, T], x + th \in C },
		\\
		\TT_C(x)
		&:=
		\varh\{\}{h \in X \mid \exists t_k \searrow 0, \exists x_k \in C \text{ such that } \frac{x_k - x}{t_k} \to h }
		,
		\\
		\TT_C^{\star}(x)
		&:=
		\varh\{\}{ h \in X \mid \exists t_k \searrow 0, \exists x_k \in C \text{ such that } \frac{x_k - x}{t_k} \weaklystar h }.
	\end{aligned}
	\end{equation*}
	Moreover, we define the  \mbox{weak-$\star$} (outer) normal cone to $x$ by
	\begin{equation*}
		\NN_C^\star (x) 
		:= \varh\{\}{ x^{\star} \in \iota(Y) \mid \forall h \in  \TT_C^\star (x) : \left \langle x^\star, h \right \rangle \leq 0 }.
	\end{equation*}
	Finally, given a $\varphi \in -\NN^{\star}_C(x)$, we define the  \mbox{weak-$\star$} critical cone by
	\begin{equation*}
		\KK_C^{\star}(x, \varphi ) := \TT_C^{\star}(x) \cap \varphi\anni .
	\end{equation*}
\end{definition}

We emphasize that $\NN_C^\star(x)$ is defined to be a subset of $\iota(Y)$ and may thus be identified with a subset of the predual space $Y$.
Note that all ``cones'' in the above are indeed cones in the mathematical sense, i.e., $h \in \RR_C(x) $ implies $\alpha h \in \RR_C(x) $ for all $\alpha \geq 0$ etc. We point out that $\RR_C(x)  \subseteq  \TT_C(x) \subseteq \TT_C^{\star}(x)$.
If $C$ is convex, then \cite[Proposition 2.55]{BonnansShapiro}
\begin{equation*}
\RR_C(x) = \mathbb{R}^+(C - x),\quad \TT_C(x)   = \mathrm{cl}(\RR_C(x))   \quad \forall x \in C.
\end{equation*}
If, in addition, $X$ is reflexive,
then we have  $\TTs_C(x) = \TT_C(x)$ by Mazur's lemma.
We remark that $\TTs_C(x)$ is in general not closed since the  \mbox{weak-$\star$} topology is not sequential on infinite-dimensional spaces,
cf.\ \cite[Example 5.9]{Mehlitz}.

\section{The Directional Curvature Functional}
\label{sec:curvature_functional}

The basic idea of our SNC and SSC approach is to not discuss the curvature
properties of the admissible set $C$ separately, i.e., independently of the
optimization problem at hand, but to develop a second-order analysis that
takes into account the gradient of the objective. To accomplish the latter, we introduce
the directional curvature functional.

\begin{definition}
	\label{def:curvature_term}
	Let $x \in C$ and $\varphi \in -\NN_C^{\star}(x)$ be given.
	The  \mbox{weak-$\star$} directional curvature functional
	$Q_C^{x,\varphi} \colon  \KK_C^{\star}(x,\varphi)  \to [-\infty, \infty]$
	associated with the triple $(x, \varphi, C)$ is defined by
	\begin{equation}
		\label{eq:weakstarcurvature}
			Q_C^{x,\varphi}( h )
			:=
			\inf\varh\{\}{
				\liminf_{k \to \infty}\left \langle \varphi, r_k \right \rangle
				\mid
				\begin{aligned}
					\{r_k\} \subset X, \{t_k\} \subset \R^+ :{}
					&
					t_k \searrow 0,
					t_k \, r_k \weaklystar 0,
					\\ &
					x + t_k \, h + \frac12 \, t_k^2 \, r_k \in C
				\end{aligned}
			}
			.
	\end{equation}
\end{definition}
In case that $X$ is finite-dimensional,
the curvature functional $Q_C^{x,\varphi}$ coincides
with a generalized derivative of second order of the indicator function
$\delta_C : X \to \{0,\infty\}$,
which is termed ``second subderivative''
in \cite[Definition~13.3]{Rockafellar}.
A similar concept in infinite dimensions is
called ``second-order epiderivative'' in
\cite[Section~1]{Do}.
Therein, it is required that the
second-order difference quotients associated with the
indicator function of the set $C$, the point $x$, and the element
$\varphi \in -\NN_C^{\star}(x)$
Mosco epi-converge,
whereas our definition
just uses the weak-$\star$ limes superior (in the sense of Kuratowski)
of the epigraphs of the difference quotients,
see also \Cref{remark_mrc} below.
The functional
$Q_C^{x,\varphi}(\cdot)$ may thus be identified with  a second-order
weak-$\star$ lower subderivative.

\Cref{def:curvature_term} can be motivated as follows: Consider an optimization problem of the form
$\min \, J(x) := \dual{\varphi}{x}$ s.t.\ $x \in C$,
where $\varphi$ satisfies $\varphi \in  -\NN_C^\star(\bar x)$ for some $\bar x \in C$, i.e.,  $\dual{\varphi}{h} \geq 0$ for all $h \in \TT_C^\star(\bar x)$.
Then $\bar x$ is a critical point
and we have to study perturbations  into the critical directions $h \in \KK_C^\star(\bar x, \varphi)$ to decide whether $\bar x$ is a local minimizer or not.
If we fix a critical direction $h \in  \KK_C^\star(\bar x, \varphi)$,
then the definition of $\TT_C^{\star}(\bar x)$ yields the existence of sequences  $\{r_k\} \subset X$, $\{t_k\} \subset \R^+$
with $t_k \searrow 0$, $\smash{t_k \, r_k \weaklystar 0}$ and $x_k := \bar x + t_k \, h + \frac12 \, t_k^2 \, r_k \in C$,
and we may calculate that
$
J(x_k) - J(\bar x) =\frac12 \, t_k^2 \, \dual{\varphi}{ r_k }
$.
This identity suggests that the limiting behavior of the dual pairing
$\dual{\varphi}{ r_k }$ between the gradient $\varphi$ of the objective $J$ and
the second-order correction $r_k$ is a decisive factor in the study of the
optimality of the critical point $\bar x$. Since the directional curvature
functional $Q_C^{\bar x,\varphi}(\cdot)$ allows to estimate the limes inferior of
precisely that quantity for all possible sequences $\{r_k\}$ and $\{t_k\} $, it
is only natural to consider it an adequate tool for the derivation of
optimality conditions. Note that, instead of looking at the accumulation
points of the corrections $r_k$, which is the idea of second-order tangent
sets, cf.\ \Cref{def:polyhedric_SOCT},
we only study accumulation points of the scalar
sequences $\dual{\varphi}{r_k}$
when working with the functional
$Q_C^{\bar x,\varphi}(\cdot)$. 
Thus, it is possible to obtain information even when the second-order
corrections $r_k$ diverge or cannot be analyzed properly. Before turning our
attention to SNC and SSC, in what follows, we first state some preliminary
results on the properties of the directional curvature functional that are
needed for our investigation.

\begin{lemma}
\label{lemma:elementaryproperties}
Let $x \in C$ be given. Then the following assertions hold.
	\label{lemma:basicproperties}
	\begin{enumerate}
		\item
			\label{item:basicproperties_1}
			For all $\varphi \in -\NN_C^{\star}(x)$, $ h \in  \KK_C^{\star}(x, \varphi)$, $\alpha > 0$, we have $Q_C^{x,\varphi}( \alpha h ) = \alpha^2 Q_C^{x,\varphi}( h )$.
		\item
			\label{item:basicproperties_2}
			If  $C$ is convex, then $Q_C^{x,\varphi}(h) \geq 0$  for all $\varphi \in -\NN_C^{\star} (x)$ and $h \in \KK_C^{\star}(x,\varphi)$.
	\end{enumerate}
\end{lemma}

\begin{proof}
Assertion 
\ref{item:basicproperties_1}
can be checked by a simple scaling argument. To prove \ref{item:basicproperties_2},
suppose that $C$ is convex, let $h \in \KK_C^\star(x,\varphi)$
and $\varphi \in -\NN_C^\star  (x)$ be arbitrary but fixed,
and let $\{r_k\}$, $\{t_k\}$ be sequences as in the definition of $Q_C^{x,\varphi}( h )$.
Then it holds $\frac{2}{t_k}  h +   r_k  \in \RR_C(x) \subseteq \TT_C^\star(x)$ due to the convexity of $C$ and, consequently,
$
\dual{\varphi}{r_k} = 
	\bigdual{\varphi}{\frac{2}{t_k}  h +   r_k } \geq 0
	$.
Taking the limes inferior for $k \to \infty$ and the infimum over all $\{r_k\}$, $\{t_k\}$ now  yields the claim.
\end{proof}

In addition to \Cref{lemma:elementaryproperties}, we  have the following  \mbox{weak-$\star$} lower semicontinuity result.

\begin{lemma}
	\label{lem:curvature_uhs}
	Let $x \in C$ and $\varphi \in - \NN_C^\star(x)$ be given. Let $\{h_n\} \subset \KK_C^\star(x,\varphi)$ be a sequence such that  $\smash{h_n \weaklystar h}$ holds for some $h \in X$ and such that there exist sequences $\{r_{n,k}\} \subset X$ and $\{t_{n,k}\} \subset \mathbb{R}^+$ and a constant $M>0$ with 
\begin{gather*}
t_{n,k} \searrow 0,\quad   t_{n,k} r_{n,k} \weaklystar 0,\quad \dual{\varphi}{r_{n,k}} \to Q_C^{x,\varphi}(h_n)\quad \text{ for all }n\text{ as } k \to \infty \text{ and }\\
 \|  t_{n,k} \, r_{n,k} \|_X \leq M,\quad x + t_{n,k} \, h_n + \frac12 \, t_{n,k}^2 \, r_{n,k} \in C\quad \text{ for all }n, k.
\end{gather*}
Then,  $h$ is an element of the critical cone $\KK_C^\star(x,\varphi)$ and it holds 
\begin{equation}
\label{eq:uhscurvatureestimate}
Q_C^{x,\varphi}(h) \leq \liminf_{n \to \infty} Q_C^{x, \varphi}(h_n).
\end{equation}
\end{lemma}
\begin{proof}
	Consider a countable dense subset $\{y_i\}$ of $Y$ and choose a sequence $k_n$ such that 
	\begin{equation}
	\label{eq:subsequenceconstruction}
		t_{n,k_n} \le \frac1n
		,
		\quad
		\dual{\varphi}{r_{n,k_n}} \le Q_C^{x,\varphi}(h_n) + \frac1n
		\quad
		\text{and}
		\quad
		\abs{\dual{y_i}{t_{n,k_n} \, r_{n,k_n}}} \le \frac1n
		\quad
		 \forall i \leq n
	\end{equation}
	holds for all $n \in \N$. Then,  our assumptions on $r_{n,k}$ imply $\norm{t_{n,k_n} \, r_{n,k_n}}_X \le M$, and we may deduce from \eqref{eq:subsequenceconstruction} that the sequences $t_n := t_{n, k_n}$ and $\smash{r_n := r_{n,k_n} + \frac2{t_n} \, (h_n - h)}$ satisfy
\begin{equation*}
t_n \searrow 0,\quad t_n r_n \weaklystar 0,\quad x + t_n \, h + \frac12 \, t_n^2 \, r_n \in C\quad \text{and}\quad \liminf_{n \to \infty} \dual{\varphi}{r_{n}}  \leq \liminf_{n \to \infty} Q_C^{x,\varphi}(h_n).
\end{equation*}
The above yields $h \in \TT_C^\star(x)$ and, since we trivially have $h \in \varphi \anni$, $h \in \KK_C^\star(x,\varphi)$. Moreover, we obtain \eqref{eq:uhscurvatureestimate} from the definition of $Q_C^{x,\varphi}(h)$. This proves the claim. 
\end{proof}

We point out that \Cref{lem:curvature_uhs} is in particular applicable if for every $h \in \KK_C^\star(x, \varphi)$ we can find sequences $\{r_k\} \subset X$ and $\{t_k\} \subset \R^+$ that realize the infimum in \eqref{eq:weakstarcurvature} with strong convergence $t_k r_k \to 0$. 
Since sets $C$ with the latter property  prove to be useful also in different contexts, we introduce the following concept.

\begin{definition}[Mosco Regularity Condition (MRC)] 
	\label{definition:MCC}
	We say that $C$ is Mosco regular in
	$(x, \varphi) \in C\times -\NN_C^{\star}(x)$ if 
			\begin{equation}
				\label{eq:MRC}
				\tag{\textup{MRC}}
				\begin{aligned}
				& \forall  h \in \KK_C^{\star}(x, \varphi) \ \exists \{r_k\} \subset X, \{t_k\} \subset \mathbb{R}^+:
				\\
				&\qquad  t_k \searrow 0 , \ t_k \, r_k \to 0 , \  x + t_k \, h + \frac12 \, t_k^2 \, r_k \in C, \ Q_C^{x,\varphi }(h) = \lim_{k \to \infty} \left \langle \varphi, r_k \right \rangle.
				\end{aligned}
			\end{equation}
\end{definition}

\begin{remark}
\label{remark_mrc}
It is easy to see that \eqref{eq:MRC} holds in $(x, \varphi) \in C \times -\NN_C^\star(x)$ if and only if  
	\begin{equation*}
			Q_C^{x,\varphi}( h )
			=
			\inf\varh\{\}{
				\liminf_{k \to \infty}\left \langle \varphi, r_k \right \rangle
				\mid
				\begin{aligned}
					\{r_k\} \subset X, \{t_k\} \subset \R^+ :{}
					&
					t_k \searrow 0,
					t_k \, r_k \to 0,
					\\ &
					x + t_k \, h + \frac12 \, t_k^2 \, r_k \in C
				\end{aligned}
			}
	\end{equation*}
for all $h \in \KK_C^{\star}(x, \varphi )$, i.e., if and only if the  functional $Q_C^{x,\varphi}( \cdot )$ remains unchanged when we replace 
the  \mbox{weak-$\star$} convergence of $t_k \, r_k$ with strong convergence. In the context of Kuratowski limits, the latter means that the \mbox{weak-$\star$} limes superior and the strong limes superior of the epigraphs of the second-order difference quotients associated with the indicator function of the set $C$, the point $x$, and the element $\varphi  \in -\NN_C^{\star}(x) $ coincide. We again refer to \cite{Do} for details. 
\end{remark}
We will see in the following section that the condition \eqref{eq:MRC} is also of significance for the study of second-order optimality conditions as it allows to weaken the regularity assumptions on the objective needed for the derivation of SNC. Note that \eqref{eq:MRC} is always satisfied when $X$ is finite-dimensional. Further conditions ensuring \eqref{eq:MRC} can be found in \Cref{lemma:polyhedricity,lem:SORCurvatureTerm}.

\section{Necessary and Sufficient Second-Order Conditions}
\label{sec:second-order_conditions}

Having introduced the directional curvature functional $Q_C^{x,\varphi}(\cdot)$, 
we now turn our attention to SNC and SSC for minimization problems of the form
\begin{equation}
	\label{eq:problem}
	\tag{\textup{P}}
	\begin{aligned}
		\text{Minimize} \quad & J(x), &
		\text{such that} \quad & x \in C.
	\end{aligned}
\end{equation}
In the remainder of this paper, when discussing optimality conditions for a
problem of the type \eqref{eq:problem}, we always require that (in
addition to our standing \Cref{asm:standing_assumption}) the following assumption holds.
\begin{assumption}
	\label{asm:standing_SOC}~
	\begin{enumerate}
		\item $\bar x $ is a fixed element of the set $C$ (the minimizer/candidate for a minimizer),
		\item $J : C \to \mathbb{R}$ is a function such that there exist a $J'(\bar x) \in \iota(Y)$ and a bounded bilinear $J''(\bar x) : X \times X \to \R$ with
			\begin{equation}
			\label{eq:hadamard_taylor_expansion}
				\lim_{k \to \infty}
				\frac{J(\bar x + t_k \, h_k) - J(\bar x) - t_k \, J'(\bar x) \, h_k - \frac12 \, t_k^2 \, J''(\bar x) \, h_k^2}{t_k^2}
				=
				0
			\end{equation}
			for all $\{h_k\} \subset X$, $\{t_k\} \subset \mathbb{R}^+$ satisfying $t_k \searrow 0$,  $h_k \weaklystar h \in X$ and $\bar x + t_k h_k \in C$.
	\end{enumerate}
\end{assumption}
Note that we use the abbreviations $ J'(\bar x) \, h  := \left \langle J'(\bar x), h \right \rangle $
and $J''(\bar x) \, h^2 := J''(\bar x) ( h,h )$ for all $h \in X$ in \eqref{eq:hadamard_taylor_expansion},
and that \eqref{eq:hadamard_taylor_expansion} is automatically satisfied if $J$ admits a second-order Taylor expansion of the form
			\begin{equation}
			\label{eq:strong_taylor_expansion}
				J(\bar x + h)
				-
				J(\bar x)
				-
				J'(\bar x) \, h
				-
				\frac12 \, J''(\bar x) \, h^2
				=
				\oo(\norm{h}^2_X)
				\quad\text{ as } \norm{h}_X \to 0.
			\end{equation}
We begin our investigation by stating necessary conditions of first order.
\begin{theorem}[First-Order Necessary Condition]
	\label{thm:FONC}
	Suppose that $\bar x$ is a local minimizer of \eqref{eq:problem}, i.e., assume that there is an $\varepsilon > 0$ with
	\begin{equation*}
		J(x) \geq J(\bar x)\qquad \forall x \in C \cap B_\varepsilon^X(\bar x),
	\end{equation*}
	where $ B_\varepsilon^X(\bar x)$ denotes the closed ball of radius $\varepsilon$ around $\bar x$.
	Then, $J'(\bar x) \in -\NN^\star_C(\bar x)$.
\end{theorem}
\begin{proof}
	Let $h \in \TT_C^\star(\bar x)$ be given.
	By definition, there are sequences $\{x_k\} \subset C$, $\{t_k\} \subset \R^+$
	with $t_k \searrow 0$ and $(x_k - \bar x)/t_k \weaklystar h$.
	We set $h_k := (x_k - \bar x) / t_k$.
	Then for large enough $k$, we have (due to  \eqref{eq:hadamard_taylor_expansion} and the boundedness of \mbox{weakly-$\star$} convergent sequences)
	\begin{equation*}
		0
		\le
		\frac{
			J(x_k)
			-
			J(\bar x)
		}{t_k}
		=
		\frac{J(\bar x + t_k \, h_k) - J(\bar x)}{t_k}
		=
		J'(\bar x) \, h_k  + \OO(t_k)
		\to
		J'(\bar x) \, h.
	\end{equation*}
	The above and our assumption $J'(\bar x) \in \iota(Y)$ yield $J'(\bar x) \in -\NN_C^\star(\bar x)$ as claimed.
\end{proof}
Using \Cref{thm:FONC}, we can provide second-order optimality conditions for \eqref{eq:problem}.

\begin{theorem}[SNC Involving the Directional Curvature Functional]
	\label{thm:SNC}
	Suppose that $\bar x$ is a local minimizer of \eqref{eq:problem} such that
	\begin{equation}
		\label{eq:second_order_growth}
		J(x)
		\ge
		J(\bar x) + \frac{c}{2} \, \norm{x - \bar x}^2_X
		\qquad\forall x \in C \cap B_\varepsilon^X(\bar x)
	\end{equation}
	holds for some $c \geq 0$ and some $\varepsilon > 0$. Assume further that one of the following conditions is satisfied.
	\begin{enumerate}
		\item The map $h \mapsto J''(\bar x) \, h^2$ is   \mbox{weak-$\star$} upper semicontinuous. \label{assumption-usc}
		\item  The admissible set $C$ satisfies   \eqref{eq:MRC}  in $(\bar x, J'(\bar x) ) \in C \times -\NN_C^\star(\bar x)$.  \label{assumption-mrc}
	\end{enumerate}
	Then
	\begin{equation}
		\label{eq:weakstarSNC}
		Q_C^{\bar x, J'(\bar x)}(h)  + J''(\bar x) \, h^2
		\ge
		c \, \|h\|_X^2
		\quad
		\forall h \in \KK_C^{\star}(\bar x, J'(\bar x)).
	\end{equation}
\end{theorem}
\begin{proof}
	 We first consider the case with  \ref{assumption-usc}: From the definition of $Q_C^{\bar x, J'(\bar x)}(\cdot)$, it follows that for every $h \in  \KK^\star_C(\bar x,J'(\bar x))$ and every $\delta > 0$ we can find sequences 
	  $\{r_k\} \subset X$ and $\{t_k\} \subset \R^+$ such that
	$t_k \searrow 0$,
	$t_k \, r_k \weaklystar 0$,
	$x_k := \bar x + t_k \, h + \frac12 \, t_k^2 \, r_k \in C$
	and 
	\begin{equation}
	\label{eq:choice_of_rk}
		\lim_{k \to \infty} \dual{ J'(\bar x)}{ r_k } \le Q_C^{\bar x, J'(\bar x)}(h) + \delta.
	\end{equation}
	Since $x_k \to \bar x$ strongly in $X$, \eqref{eq:second_order_growth} entails
	$J(x_k) \ge J(\bar x) + \frac{c}{2} \, \norm{x_k - \bar x}^2_X $ for  large enough $k$.
	Hence, we may use \eqref{eq:hadamard_taylor_expansion}, $J'(\bar x) h = 0$ and the  \mbox{weak-$\star$} lower semicontinuity of the norm $\|\cdot\|_X$ to obtain 
	\begin{equation}
	\label{eq:randomestimate2}
	\begin{aligned}
		0
		&=
		\lim_{k \to \infty}
		\frac{J(x_k) - J(\bar x) - t_k \, J'(\bar x) \left ( h + \frac12 \, t_k r_k \right )  - \frac12 \, t_k^2 \, J''(\bar x) \left ( h + \frac12 \, t_k r_k\right ) ^2}{t_k^2}
		\\
		&\ge
		\limsup_{k \to \infty}
		\frac{\frac{c}{2} \,t_k^2  \norm{ h + \frac12 \, t_k \, r_k}^2_X - t_k \, J'(\bar x) \left (\frac12 \, t_k r_k \right )  - \frac12 \, t_k^2 \, J''(\bar x) \left ( h + \frac12 \, t_k r_k\right ) ^2}{t_k^2}
		\\&
		\geq 
		\frac{c}{2} \|h\|_X^2
		-\frac12\,\liminf_{k \to \infty}
		\dual{ J'(\bar x)}{ r_k }
		-\frac12\,\limsup_{k \to \infty}
		\bigh(){
			J''(\bar x) \left ( h + \textstyle \frac12 \, t_k r_k\right ) ^2
		}.
	\end{aligned}
	\end{equation}
	From \eqref{eq:choice_of_rk}, \eqref{eq:randomestimate2} and the  \mbox{weak-$\star$} upper semicontinuity of $h \mapsto J''(\bar x) \, h^2$, it  follows
	\begin{equation*}
		c \|h\|_X^2
		\le
		Q_C^{\bar x, J'(\bar x)}(h) + \delta
		+
		J''(\bar x) \, h^2
		.
	\end{equation*}
	Passing to the limit $\delta \searrow 0$ in this inequality yields the claim in the first case. 

\pagebreak[2]

	It remains to prove \eqref{eq:weakstarSNC} under assumption \ref{assumption-mrc}. To this end, we note that, if \eqref{eq:MRC}  holds in $(\bar x, J'(\bar x) ) \in C \times -\NN_C^\star(\bar x)$, then for every $h \in  \KK^\star_C(\bar x,J'(\bar x))$ we can find 
	  $\{r_k\} \subset X$, $\{t_k\} \subset \R^+$ such that
	$t_k \searrow 0$,
	$t_k \, r_k \to 0$,
	$x_k := \bar x + t_k \, h + \frac12 \, t_k^2 \, r_k \in C$
	and 
	\begin{equation*}
		\lim_{k \to \infty} \dual{ J'(\bar x)}{ r_k } = Q_C^{\bar x, J'(\bar x) }(h).
	\end{equation*}
	Now,
	the second-order condition \eqref{eq:weakstarSNC} follows
	analogously to case \ref{assumption-usc}.
\end{proof}

\begin{theorem}[SSC Involving the Directional Curvature Functional]
	\label{thm:SSC_wo}
	Assume that the map $h \mapsto J''(\bar x) \, h^2$
	is  \mbox{weak-$\star$} lower semicontinuous, that 
	$J'(\bar x) \in -\NNs_C(\bar x)$ and that
	\begin{equation}
		\label{eq:SSC_wo}
		Q_C^{\bar x, J'(\bar x)}(h)  + J''(\bar x) \, h^2
		>
		0
		\quad
		\forall h \in \KKs_C(\bar x, J'(\bar x)) \setminus \{0\}.
	\end{equation}
	Suppose further that
	\begin{equation} 
		\label{eq:NDS}
		\tag{\textup{NDC}}
		\begin{aligned}
			&\text{for all $\{h_k\} \subset X$, $\{t_k\} \subset \R^+$ with $\bar x + t_k \, h_k \in C$, $h_k \weaklystar 0$, $t_k \searrow 0$}
 			\\[-0.1cm]
			&\text{and $\norm{h_k}_X = 1$, it is true that }\\
			&\qquad\qquad\qquad\qquad\quad \  \liminf_{k \to \infty} \bigh(){
				\dual{J'(\bar x)}{h_k / t_k}
				+
				\frac12 \, J''(\bar x) \, h_k^2
			}
			> 0
			.		
		\end{aligned}
	\end{equation}
	Then $\bar x$ satisfies the growth condition \eqref{eq:second_order_growth} with some constants $c > 0$ and $\varepsilon > 0$.
\end{theorem}
\begin{proof}
	We argue by contradiction.
	Assume that there are no $c>0$, $\varepsilon > 0$ such that \eqref{eq:second_order_growth} holds.
	Then there are sequences $\{x_k\} \subset C$ and $\{c_k\} \subset \mathbb{R}^+$
	such that
	\begin{equation*}
		c_k \searrow 0
		,\quad
		\norm{x_k - \bar x }_X \to 0
		,\quad\text{and}\quad
		J(x_k) < J(\bar x) + \frac {c_k}2 \, \norm{x_k - \bar x}_X^2.
	\end{equation*} 
	Define $t_k := \norm{x_k - \bar x}_X$
	and
	$h_k := (x_k - \bar x) / t_k$.
	Then $\norm{h_k}_X = 1$ for all $k$ and we may extract a subsequence (not relabeled)
	such that $h_k \weaklystar h \in \TT_C^{\star}(\bar x)$.
	From \eqref{eq:hadamard_taylor_expansion}, it follows 
	\begin{equation}
	\label{eq:randomestimate42}
	\begin{aligned}
		0
		&=
		\lim_{k \to \infty} \frac{J(\bar x + t_k \, h_k) - J(\bar x) - t_k \, J'(\bar x) \, h_k - \frac12 \, t_k^2 \, J''(\bar x) \, h_k^2}{t_k^2}
		\\&
		\le
		\liminf_{k \to \infty} \frac{\frac{c_k}{2} \, \norm{x_k - \bar x}_X^2 - t_k \, J'(\bar x) \, h_k - \frac12 \, t_k^2 \, J''(\bar x) \, h_k^2}{t_k^2}
		\\&
		=
		\liminf_{k \to \infty} \frac{-J'(\bar x) \, h_k - \frac12 \, t_k \, J''(\bar x) \, h_k^2}{t_k}
		\\&
		\le
		-\limsup_{k \to \infty} \frac{J'(\bar x) \, h_k}{t_k}
		-
		\liminf_{k \to \infty} \frac{\frac12 \, t_k \, J''(\bar x) \, h_k^2}{t_k}
		\\&
		\le
		-\limsup_{k \to \infty} \frac{J'(\bar x) \, h_k}{t_k}
		-
		\frac12 \, J''(\bar x) \, h^2.
	\end{aligned}
	\end{equation}
	Thus, $\limsup_{k\to\infty} J'(\bar x) \, h_k / t_k$ is bounded from
	above and from $t_k \searrow 0$ we infer $\limsup_{k\to\infty} J'(\bar x) \, h_k \le 0$.
	Together with $J'(\bar x) \in -\NNs_C(\bar x)$
	we find
	$J'(\bar x) \, h = 0$.
	Using
	$h \in \TTs_C(\bar x)$,
	this yields
	$h \in \KKs_C(\bar x,J'(\bar x))$.
	Setting
	$r_k := 2 \, (h_k - h)/t_k$ and using \eqref{eq:randomestimate42},
	$J'(\bar x) \, h = 0$ and \Cref{def:curvature_term},
	we obtain 
	\begin{align*}
		J''(\bar x) \, h^2
		&\le
		-2 \, \limsup_{k \to \infty} \frac{J'(\bar x) \, h_k}{t_k}
		=
		- \limsup_{k \to \infty} J'(\bar x) \, r_k
		\le
		- \liminf_{k \to \infty} J'(\bar x) \, r_k
		\\&
		\le
		- Q_C^{\bar x, J'(\bar x)}(h).
	\end{align*}
	From \eqref{eq:SSC_wo}, we may now deduce that  $h$ is zero. This is a contradiction with \eqref{eq:NDS}, see  the properties of the sequences $\{h_k\}$, $\{t_k\}$ and \eqref{eq:randomestimate42}.
\end{proof}

The acronym \eqref{eq:NDS}
stands for ``non-degeneracy condition''.
Comments on  \eqref{eq:NDS}
are provided in \Cref{subsec:NDS}. 
By combining the previous two theorems,
we arrive at our main theorem on no-gap second-order conditions.
\begin{theorem}[No-Gap Second-Order Optimality Condition]
	\label{thm:no-gap-SOC}
	Assume that the map $h \mapsto J''(\bar x) \, h^2$
	is  \mbox{weak-$\star$} lower semicontinuous, that $J'(\bar x) \in -\NNs_C(\bar x)$, that  \eqref{eq:NDS} holds, and that one of the conditions \ref{assumption-usc} and \ref{assumption-mrc} in \Cref{thm:SNC} is 		satisfied.  
	Then, the condition
	\begin{equation*}
		Q_C^{\bar x, J'(\bar x)}(h)  + J''(\bar x) \, h^2
		>
		0
		\quad
		\forall h \in \KKs_C(\bar x, J'(\bar x)) \setminus \{0\}
	\end{equation*}
	is equivalent to the quadratic growth condition \eqref{eq:second_order_growth} with constants $c>0$ and $\varepsilon > 0$.
\end{theorem}

Some remarks regarding \Cref{thm:SNC,thm:SSC_wo,thm:no-gap-SOC} are in order.

\begin{remark}\hfill
	\begin{enumerate}
	\item Note that \Cref{thm:SNC} yields that there are two ways to obtain the second-order necessary condition \eqref{eq:weakstarSNC} in the situation of  \Cref{asm:standing_SOC}: We can assume either that the second derivative of $J$ at $\bar x$ has additional (semi)\-continuity properties (this is case \ref{assumption-usc}) or that the set $C$ is sufficiently well-behaved at $\bar x$ (this is case \ref{assumption-mrc}). The observation that one has the choice to impose additional assumptions either on the appearing sets or the appearing functions can be made frequently when working with the directional curvature functional.
	\item We point out that the regularity condition \eqref{eq:MRC} is not helpful in the derivation of the SSC in \Cref{thm:SSC_wo}. 
At the heart of the proof of \Cref{thm:SSC_wo}
is, after all, the contradiction argument which only provides a  \mbox{weak-$\star$}
convergent subsequence.
	\item It is easy to see that the differentiability assumptions on $J$ in \Cref{asm:standing_SOC} can be weakened when one is interested in only one of the results in  \Cref{thm:SNC,thm:SSC_wo}. To derive the necessary condition \eqref{eq:weakstarSNC} in case \ref{assumption-mrc} of \Cref{thm:SNC}, for example, it suffices to assume that
	\begin{equation*} 
		\mspace{-64mu}
	\begin{aligned}
	&h \in X, \{r_k\} \subset X,  \ \{t_k\} \subset \R^+, t_k \searrow 0, t_k \, r_k \to 0 \text{ in X}  \\
	&\Rightarrow  \max \left( 0,   \frac{J(\bar x + t_k h + \frac12 t_k^2 r_k) - J(\bar x) - t_k J'(\bar x)(h + \frac{1}{2}t_k r_k) - \frac12 t_k^2 J''(\bar x)h^2}{t_k^2}\right ) \to 0
	\end{aligned}
	\end{equation*}
	holds for some bounded, linear mapping $J'(\bar x) : X \to \R$ and some bilinear form $J''(\bar x) : X \times X \to \R$.  We refrain from stating the minimal differentiability properties in each of the \Cref{thm:SNC,thm:SSC_wo,thm:no-gap-SOC}  since this would just obscure the basic ideas of our analysis. 
	\item We point out that it is possible to modify the proofs of \Cref{thm:SNC,thm:SSC_wo} to obtain second-order conditions in the setting of two-norms discrepancy, i.e., in the situation where the objective $J$ only satisfies a second-order Taylor expansion à la \eqref{eq:strong_taylor_expansion} w.r.t.\ some norm $\|\cdot\|_Z$ that is stronger than $\|\cdot\|_X$. We leave it to the interested reader to work out this easy extension of our analysis in detail. 

\item In the finite-dimensional setting, results similar to  \Cref{thm:SNC,thm:SSC_wo,thm:no-gap-SOC} have been obtained in \cite[Theorem 13.24]{Rockafellar}.
	\end{enumerate}
\end{remark}

\section{How to Apply and Interpret the Results of \texorpdfstring{\Cref{sec:second-order_conditions}}{Section~\ref{sec:second-order_conditions}} in the Context of the Classical Theory}
\label{sec:5}
In this section,
we comment on the verification and interpretation of the conditions \eqref{eq:NDS} and \eqref{eq:MRC} appearing in our second-order conditions,
see \Cref{subsec:NDS}, \Cref{lemma:polyhedricity}, and \Cref{lem:SORCurvatureTerm}.
Further, we address the computation of the directional curvature functional for polyhedric and second-order regular sets, see \Cref{subsec:PolyhedricitySORCurvature}.
Finally, we compare our theorems of \Cref{sec:second-order_conditions} with classical results.
In \Cref{subsec:PolyhedricitySORSecondOrderConditions}, we use the observations of \Cref{subsec:NDS,subsec:PolyhedricitySORCurvature} to state two corollaries of \Cref{thm:no-gap-SOC} that reproduce (and slightly extend) classical results found, e.g., in \cite{BonnansShapiro} and \cite{Casas2015}. We conclude this section with two examples that illustrate the usefulness of \Cref{thm:SNC,thm:no-gap-SOC,thm:SSC_wo}.

\subsection{Remarks on the Non-Degeneracy Condition \texorpdfstring{\eqref{eq:NDS}}{(\ref{eq:NDS})}}
\label{subsec:NDS}

The condition \eqref{eq:NDS} appearing in our SSC can be interpreted as a generalized Legendre condition,
cf.\ \cite[Section 3.3.2]{BonnansShapiro}
and \Cref{lemma:sufficientfornondegeneracy}~\ref{item:sufficientfornondegeneracy_2}. In contrast to a classical Legendre condition,
our requirement
\eqref{eq:NDS} is a condition on the interplay between the curvature of $C$ in $\bar x$
in the direction $J'(\bar x)$ and the properties of $J''(\bar x)$. In practice, \eqref{eq:NDS} can be ensured, e.g., by assuming ellipticity of the second derivative $J''(\bar x)$ or by assuming that the admissible set $C$ has ``positive'' curvature at $\bar x$ (or some combination of the both). Some sufficient criteria can be found in the following lemma.

\begin{lemma}
	\label{lemma:sufficientfornondegeneracy}
	Each of the following conditions is sufficient for \eqref{eq:NDS}.
	\begin{enumerate}
		\item
			\label{item:sufficientfornondegeneracy_1}
			$C$ is a  convex subset of $X$,
			$J'(\bar x) \in -\NNs_C(\bar x)$,
			and $J''(\bar x)h^2 \geq r > 0$ for all $h \in X$ with $\|h\|_X= 1$.
		\item
			\label{item:sufficientfornondegeneracy_2}
			$C$ is a convex subset of $X$,
			$J'(\bar x) \in -\NNs_C(\bar x)$,
			and $J''(\bar x)$ is a Legendre form
			in the sense that
			$h \mapsto J''(\bar x) \, h^2$ is  \mbox{weak-$\star$} lower semicontinuous
			and
			\begin{equation*}
				h_k \weaklystar h \text{ and } J''(\bar x) \, h_k^2 \to J''(\bar x) \, h^2
				\quad\Rightarrow\quad
				h_k \to h
			\end{equation*}
			holds for all sequences $\{h_k\} \subset X$.
		\item
			\label{item:sufficientfornondegeneracy_3}
			There exist $c, \varepsilon > 0$
			with
			$J'(\bar x) \, (x - \bar x) \ge \frac{c}{2} \, \norm{x - \bar x}^2$
			for all $x \in B^X_\varepsilon(\bar x) \cap C$
			and
			the map $h \mapsto J''(\bar x)h^2$ is  \mbox{weak-$\star$} lower semicontinuous.
		\item
			\label{item:sufficientfornondegeneracy_4}
			$X$ is finite-dimensional.
	\end{enumerate}
\end{lemma}
\begin{proof}
	Case \ref{item:sufficientfornondegeneracy_4} is trivial. To prove \eqref{eq:NDS} in the cases \ref{item:sufficientfornondegeneracy_1} to \ref{item:sufficientfornondegeneracy_3}, we have to show that for all 
	$\{t_k\}$, $\{h_k\}$ 
	satisfying $t_k \searrow 0$, $\smash{h_k \weaklystar 0}$ and  $\norm{h_k}_X = 1$, $\bar x + t_k \,h_k \in C$ for all $k$, it holds
	$\liminf_{k \to \infty} \bigh(){ \dual{J'(\bar x)}{h_k / t_k} + \frac12 \, J''(\bar x) \, h_k^2 } > 0$.
	Therefore, we assume that sequences $\{t_k\}$, $\{h_k\}$ with the above properties are given.

	We first discuss the
	case that $C$ is convex.
	For such a $C$, it holds $h_k \in \RR_C(\bar x)$ and, as a consequence, 
	$J'(\bar x) \, h_k \ge 0$, so it suffices to prove
	$\liminf_{k \to \infty} J''(\bar x) \, h_k^2  > 0$ to obtain the claim.
	The latter is trivially true under assumption \ref{item:sufficientfornondegeneracy_1}.
	Similarly,
	in
	case~\ref{item:sufficientfornondegeneracy_2},
	we have
	$\liminf_{k \to \infty} J''(\bar x) \, h_k^2 \ge 0$
	by lower semicontinuity.
	Moreover, if we had
	$J''(\bar x) \, h_k^2 \to 0$ (along a subsequence),
	we would obtain the contradiction $h_k \to 0$ (along a subsequence).
	Hence,
	$\liminf_{k \to \infty} J''(\bar x) \, h_k^2 > 0$ and \eqref{eq:NDS} is proved for \ref{item:sufficientfornondegeneracy_2}.

	In case
	\ref{item:sufficientfornondegeneracy_3},
	we have
	\begin{equation*}
		\dual{J'(\bar x)}{h_k / t_k}
		=
		t_k^{-2} \, \dual{J'(\bar x)}{t_k \, h_k}
		\ge
		\frac{c}{2} \, t_k^{-2} \, \norm{t_k \, h_k}^2
		=
		\frac{c}{2}
	\end{equation*}
	for $k$ large enough.
	Together with
	$\liminf_{k \to \infty} J''(\bar x) \, h_k^2 \ge J''(\bar x) \, 0^2 = 0$,
	the above yields
	the desired inequality
	$\liminf_{k \to \infty} \bigh(){ \dual{J'(\bar x)}{h_k / t_k} + \frac12 \, J''(\bar x) \, h_k^2 } > 0 $. This completes the proof.
\end{proof}

Note that in case \ref{item:sufficientfornondegeneracy_3} of \Cref{lemma:sufficientfornondegeneracy}, the second derivative $J''(\bar x)$ is allowed to be negative definite. This effect occurs whenever the curvature of the set $C$ at $\bar x$ is such that it can compensate for negative curvature of the objective $J$. In \Cref{subsec:bangbang}, we will see an example, where the latter actually happens. We further point out that, under the assumptions of \Cref{lemma:sufficientfornondegeneracy}~\ref{item:sufficientfornondegeneracy_3},
the directional curvature functional is coercive in the sense that
\begin{equation}
\label{eq:coercivecurvature}
	Q_C^{\bar x, J'(\bar x)}(h) \ge c \, \norm{h}_X^2
	\qquad\forall h \in \KKs_C(\bar x, J'(\bar x)),
\end{equation}
where $c$ is the constant in
\Cref{lemma:sufficientfornondegeneracy}~\ref{item:sufficientfornondegeneracy_3}.
Indeed, for
$h \in \KKs_C(\bar x, J'(\bar x))$,
$t_k \searrow 0$, $t_k \, r_k \weaklystar 0$ and $x_k := \bar x + t_k \, h + \frac12 \, t_k^2 \, r_k \in C$,
we have
\begin{equation*}
\begin{aligned}
	\liminf_{k \to \infty} J'(\bar x) \, r_k
	&=
	\liminf_{k \to \infty} J'(\bar x) \, \frac{x_k - \bar x}{\frac12 \, t_k^2}
	\\
	&\ge
	\liminf_{k \to \infty} c\,\frac{\norm{x_k - \bar x}_X^2}{t_k^2}
	=
	\liminf_{k \to \infty} c\,\norm{h + \frac12 \, t_k \, r_k}_X^2
	\ge
	c\,\norm{h}_X^2.
\end{aligned}
\end{equation*}
The above gives a lower estimate for the directional curvature functional that is often also interesting for its own sake. We will get back to this topic in \Cref{subsec:bangbang}. 

\subsection{Directional Curvature, Polyhedricity and Second-Order Regularity}
\label{subsec:PolyhedricitySORCurvature}

In this section, we discuss how the notion of directional curvature is related
to the classical concepts of polyhedricity and second-order regularity, and how
these properties can be used to calculate the functional
$Q_C^{x, \varphi}(\cdot)$
for a given tuple  $(x, \varphi) \in C \times -\NN_C^{\star}(x)$. 
Recall the following definitions,
cf.\ \cite[Section 3.2.1, Definitions~3.51, 3.85]{BonnansShapiro}, \cite{Mignot}, and \cite[Lemma 4.1]{WachsmuthPolyhedricity}.
\begin{definition}[Polyhedricity and Second-Order Regularity]~
	\label{def:polyhedric_SOCT}
	\begin{enumerate}
		\item
			The set $C$ is said to be polyhedric at $x \in C$
			if $C$ is convex and
			\begin{equation*}
				\TT_C(x) \cap \varphi^\perp = \mathrm{cl}\bigh(){ \RR_C(x) \cap  \varphi^\perp }
				\quad \forall \varphi \in X^\star.
			\end{equation*}
		\item
			The strong (outer) second-order tangent set to a tuple $(x,h) \in C \times \TT_C(x)$
			is given by
			\begin{equation*}
				\TT_C^{2}(x,h)
				:=
				\Bigh\{\}{
					r \in X \mid \exists  t_n \searrow 0, \dist\bigh(){ x + t_n\,h + {\textstyle\frac{1}{2}}\,t_n^2 \, r, C}
					=
					\oo(t_n^2)
				}.
			\end{equation*}
		\item
		\label{def:polyhedric_SOCT:iii}
			The set $C$ is called (outer) second-order regular at  $x \in C$ if for all $h \in  \TT_C(x)$ and all $x_n \in C$ of the form $x_n := x + t_n h + \frac{1}{2} t_n^2 r_n$ with $t_n \searrow 0$ and $t_n r_n \to 0$ it holds
			\begin{equation*}
				\lim_{n \to \infty} \mathrm{dist}\left (r_n, \TT_C^{2}(x,h)  \right ) =0.
			\end{equation*}
	\end{enumerate}
\end{definition}
Note that 
sequences $\{x_n\}$, $\{r_n\}$ as in \ref{def:polyhedric_SOCT:iii}
exist for all $h \in \TT_C(x)$ and 
that a set $C$ can only be (outer) second-order regular at $x \in C$
if $\TT^2_C(x,h) \ne \emptyset$
for all $h \in \TT_C(x)$ (since otherwise $\dist(\cdot, \TT_C^2(x,h)) = +\infty$ by the usual conventions).

First, we show that the boundary of polyhedric sets is not curved.

\begin{lemma}[Curvature of Polyhedric Sets]
	\label{lemma:polyhedricity}
		Assume that $X$ is reflexive and that $C$ is polyhedric at $x \in C$. Then \eqref{eq:MRC} is satisfied in $(x, \varphi)$ for all $\varphi \in -\NN^\star_C(x)$ and
		\begin{equation*} 
			Q_C^{x,\varphi}(h) =
			0
			\qquad\forall h \in \KK_C^\star (x, \varphi).
		\end{equation*}
\end{lemma}

\begin{proof}
From the reflexivity, Mazur's lemma and the convexity and closedness of $C$,
it follows $\TT^\star_C(x)  = \TT_C(x)$.
Consequently, $ \KK_C^\star (x, \varphi) = \TT_C(x) \cap \varphi^\perp$ for every $\varphi \in -\NN_C^{\star}(x)$.
For $h \in \RR_C(x) \cap \varphi\anni$ the choice $t_k = 1/k$, $k \in \mathbb{N}$ sufficiently large, and $r_k = 0$ shows
$Q^{x,\varphi}_C(h) = 0$, see also \Cref{lemma:basicproperties}~\ref{item:basicproperties_2}.
Now, let $h \in \KK_C^\star(x,\varphi) = \TT_C(x) \cap \varphi^\perp$ be given.
Owing to the polyhedricity of $C$ at $x$,
there exists a sequence  $\{h_k\} \subset  \RR_C(x) \cap  \varphi^\perp$
with $h_k \to h$  in $X$.
Now, we can apply \Cref{lemma:basicproperties}~\ref{item:basicproperties_2} and \Cref{lem:curvature_uhs}
(with $M = 0$)
to obtain
\begin{equation*}
	0  \leq Q_C^{x,\varphi}( h )
	\leq 
	\liminf_{k \to \infty} Q_C^{x,\varphi}( h_k )
	= 0
	.
\end{equation*}
The recovery sequence in \eqref{eq:MRC} can be constructed straightforwardly from $\{h_k\}$.  
\end{proof}

\Cref{lemma:polyhedricity} shows that, for a polyhedric set $C$, the directional curvature functional is always identical zero.
This is, of course, exactly what one would expect  (cf.\ also with \cite[Example 2.10]{Do} in this context). The situation is different
when $C$ possesses curvature in the sense of second-order regularity -- a
concept that is promoted and extensively used  in \cite{BonnansShapiro}.
The curvature of second-order regular sets is addressed in the next lemma.

\begin{lemma}[Curvature of Second-Order Regular Sets]
	\label{lem:curvature_of_OSOR}
Assume that $C$ is outer second-order regular at $x \in C$, that $\varphi \in -\NN^\star_C(x)$,  and that  \eqref{eq:MRC} is satisfied in $(x, \varphi)$. Then
\begin{equation*}
	Q_C^{x,\varphi}( h ) =
	\inf_{ r \in  \TT_C^{2}(x,h)}  \left \langle \varphi, r \right \rangle
	\qquad
	\forall h \in \KK^\star_C(x,\varphi).
\end{equation*}
\end{lemma}
\begin{proof}
Let $h \in \KK^\star_C(x,\varphi)$
be given. Then \eqref{eq:MRC} yields that we can find sequences $\{r_k\} \subset X$, $\{t_k\} \subset \R^+$ such that
			$t_k \searrow 0$,
			$t_k \, r_k \to 0$,
			$x + t_k \, h + \frac12 \, t_k^2 \, r_k \in C$
			and 
			\begin{equation*}
				Q_C^{x,\varphi}(h) = \lim_{k \to \infty} \left \langle \varphi,  r_k \right \rangle.
			\end{equation*}
Note that the above implies in particular that $h \in \TT_C(x)$.
From the outer second-order regularity of $C$ in $x$, we now obtain
\begin{equation*}
\lim_{k \to \infty} \dist\left (r_k, \TT_C^{2}(x,h)  \right ) =0,
\end{equation*}
i.e., there exists a sequence $\{\tilde{r}_k\} \subset \TT_C^{2}(x,h)$ with 
			\begin{equation}
				\label{eq:in_proof_pullback}
				Q_C^{x,\varphi}(h) = \lim_{k \to \infty} \left \langle \varphi,  r_k \right \rangle  = \lim_{k \to \infty} \left \langle \varphi,  \tilde{r}_k \right \rangle \geq \inf_{ r \in  \TT_C^{2}(x,h)}  \left \langle \varphi, r \right \rangle.
			\end{equation}

\pagebreak[2]

If, on the other hand, $r \in \TT_C^{2}(x,h)$ is arbitrary but fixed, then we know that there are sequences $\{t_k\} \subseteq \mathbb{R}^+$, $\{s_k\} \subseteq X$ with
\begin{equation*}
t_k \searrow 0,\quad x + t_k h + \frac{1}{2}t_k^2 ( r  + s_k ) \in C,\quad s_k \to 0,
\end{equation*}
and we obtain from the definition of $Q_C^{x,\varphi}(h)$ that
	\begin{equation*}
		\begin{aligned}
			 Q_C^{x,\varphi}( h )  
			&\leq
			\inf\varh\{\}{
				\liminf_{k \to \infty}\left \langle \varphi,  r_k \right \rangle
				\mid
				\begin{aligned}
					\{r_k\} \subset X, \{t_k\} \subset \R^+ :{}
					&
					t_k \searrow 0,
					t_k \, r_k \to 0,
					\\ & \;
					x + t_k \, h + \frac12 \, t_k^2 \, r_k \in C
				\end{aligned}
			}~\\
			&\leq  \liminf_{k \to \infty}\left \langle \varphi,  r + s_k \right \rangle
			= \left \langle \varphi,  r \right \rangle.
		\end{aligned}
	\end{equation*}
This yields 
	\begin{equation*}
		Q_C^{x,\varphi}(h) \leq \inf_{ r \in  \TT_C^{2}(x,h)}  \left \langle \varphi, r \right \rangle
	\end{equation*}
which, together with \eqref{eq:in_proof_pullback}, proves the claim.
\end{proof}

We would like to point out that \Cref{lemma:polyhedricity} cannot be obtained as a corollary of \Cref{lem:curvature_of_OSOR}. The reason for this is that polyhedric sets do not necessarily have to be second-order regular.
In fact, we have the following result.

\begin{lemma}[Necessary Condition for Second-Order Regularity]
	\label{lem:nec_cond_osor}
	Assume that  $C$ is outer second-order regular at $x \in C$.
	Then, for all $h \in \TT_C(x)$,
	it holds $\TT_C^2(x,h) \ne \emptyset$
	and
	there exists a positive sequence $t_k \searrow 0$
	such that
	$\dist(x + t_k \, h, C) = \OO(t_k^2)$ as $k \to \infty$.
\end{lemma}
\begin{proof}
	Let $h \in \TT_C(x)$ be given.
	As explained after \Cref{def:polyhedric_SOCT},
	the outer second-order regularity of $C$
	at $x \in C$
	implies
	$\TT_C^2(x,h) \neq \emptyset$.

	Now, let $r \in \TT_C^2(x,h)$ be given.
	By definition, there is a positive sequence $t_k \searrow 0$
	such that
	$ \dist\bigh(){ x + t_k \, h + \frac12 \, t_k^2 \, r, C } = \oo(t_k^2)$.
	Using the latter and the triangle inequality, we obtain
	\begin{equation*}
		\dist(x + t_k \, h, C)
		\le
		\dist\Bigh(){ x + t_k \, h + \frac12 \, t_k^2 \, r, C }
		+
		\frac12 \, t_k^2 \, \norm{r}_X
		=
		\oo(t_k^2) + \OO(t_k^2)
		=
		\OO(t_k^2).
	\end{equation*}
	This finishes the proof.
\end{proof}

Using \Cref{lem:nec_cond_osor}, we can prove that even the most elementary examples of infinite-dimensional polyhedric sets can lack the property of second-order regularity.

\begin{example}
	\label{ex:polyhedric_not_osor}
	Let $(0,1)$ be equipped with Lebesgue's measure.
	Define $X := L^2(0,1)$ and $C := \{v \in L^2(0,1) : v \ge 0\}$, and let $x$ be the unique element of $X$ with $x(\xi) = 1$ for a.a.\ $\xi \in (0,1)$. 
	Then,  $C$ is polyhedric at $x$ and it holds $\TT_C(x) = L^2(0,1)$.
	Consider now an arbitrary but fixed $\alpha \in (-1/2, -1/4)$ and let $h \in L^2(0,1)$ be defined via $h(\xi) = -\xi^\alpha$ for a.a.\ $\xi \in (0,1)$.
	Then, for all sequences $t_k \searrow 0$, it holds
	\begin{equation*}
		\dist( x + t_k \, h, C )
		=
		\Bigh(){ \int_0^{t_k^{-1/\alpha}} (1 - t_k \, \xi^\alpha)^2 \, \d\xi }^{1/2}
		=
		c_\alpha^{1/2}  \, t_k^{-1/(2\,\alpha)}
		\ne
		\OO(t_k^2)
		,
	\end{equation*}
	where $c_\alpha = \frac{2 \, \alpha^2}{(\alpha+1)\,(2\,\alpha+1)} > 0$.
	Hence, $\TT_C^2(x,h) = \emptyset$ and, consequently, $C$ cannot be outer second-order regular at $x$ by \Cref{lem:nec_cond_osor}.
	Using similar arguments,
	we can show that $C$ fails to be outer second-order regular at all $x \in C \setminus \{0\}$.
\end{example}

Note that, due to the effects appearing in \Cref{ex:polyhedric_not_osor}, the
concept of second-order regularity is typically not suited for the analysis of
optimal control problems with pointwise control or state constraints. It is,
however, quite useful when the optimization problem at hand is finite-dimensional or involves constraints of the form $G(x) \in K$, where $K \subset \R^d$ is a
closed, convex, non-empty set, cf.\ \Cref{ex:state_comstraint}. To simplify the derivation
of second-order optimality conditions for problems  of the latter type, we
provide a calculus rule for the curvature of preimages.

\begin{lemma}
\label{lem:SORCurvatureTerm}
Let $Z$ be a Banach space.
Assume that $C=G^{-1}(K)$ holds for some twice continuously Fréchet
differentiable function $G : X \to Z$ and
some closed, convex, non-empty set $K \subseteq Z$. Suppose further that a tuple $(x, \varphi) \in C \times - \NN_C^\star(x)$  is given such that the Zowe-Kurcyusz constraint qualification
	\begin{equation}
		\label{eq:ZKCQ}
		\tag{\textup{ZKCQ}}
		G'( x) \, X - \RR_K(G( x)) = Z
	\end{equation}
is satisfied in $x$, such that $K$ is second-order regular in $G(x)$, and such that the maps 
$h \mapsto G'(x)h $ and  $h \mapsto G''(x)h^2$ are  \mbox{weak-$\star$}-to-strong continuous. 
Then, $C$ satisfies \eqref{eq:MRC} in $(x, \varphi)$, it holds  $\KK_C^\star(x,\varphi)= \varphi^\perp \cap G'(x)^{-1} \TT_K(G( x))$, and for every $h \in \KK_C^\star(x,\varphi)$ it is true that 
\begin{equation}
\label{eq:pullbackformula}
\begin{aligned}
		Q_C^{ x , \varphi}( h) 
		&= \inf_{ r \in G'(x)^{-1}\left (\TT_K^2(G(x), G'(x)h) - G''(x)h^2\right )}  \left \langle \varphi, r \right \rangle.
\end{aligned}
\end{equation}
\end{lemma}

\begin{proof}
The proof of \Cref{lem:SORCurvatureTerm} follows the lines of that of \cite[Proposition 3.88]{BonnansShapiro}: 
Let $ h \in \KK^\star_C(x,\varphi)$ be fixed, and let $\{r_k\} \subset X, \{t_k\} \subset \R^+$ be sequences satisfying $t_k \searrow 0$, $\smash{t_k \, r_k \weaklystar 0}$ and $\smash{x + t_k \, h + \frac12 \, t_k^2 \, r_k \in C}$, i.e., $G(x + t_k \, h + \frac12 \, t_k^2 \, r_k ) \in K$.
Then, by a Taylor expansion of $G$, cf.\ \cite[Theorem~5.6.3]{Cartan1967},
and using our assumptions on $G''(x)$, we obtain
\begin{equation}
\label{eq:long_Taylor}
\begin{aligned}
	&G\bigh(){ x + t_k \, h + \frac12 \, t_k^2 \, r_k }
	\\
	&\qquad= G(x)  + t_k G'(x)h + \frac{1}{2}t_k^2  \Bigh(){ G' (x)r_k  + G''(x) \bigh(){ h + \frac12 t_k \, r_k }^2 } + \oo(t_k^2)
	\\
	&\qquad= G(x)  + t_k G'(x)h + \frac{1}{2}t_k^2  \bigh(){ G' (x)r_k  + G''(x) h^2 + \varphi_k } \in K
\end{aligned}
\end{equation}
with some $\varphi_k$ satisfying $\varphi_k \to 0$ in $Z$. Since $G'(x)$ is  \mbox{weak-$\star$}-to-strong continuous, \eqref{eq:long_Taylor} implies $ G' (x)h \in \TT_K(G(x))$ and, by 
\cite[Corollary 2.91]{BonnansShapiro}, $h \in \TT_C(x) = G'(x)^{-1} \TT_K(G( x))$. This yields $\KK_C^\star(x,\varphi)= \varphi^\perp \cap G'(x)^{-1} \TT_K(G( x))$ and shows  that it makes sense to use the second-order tangent sets $\TT_C^2(x, h)$ and $\TT_K^2(G(x), G'(x)h)$ in the following. Next, we will prove that
$\dist(r_k, \TT_C^2(x, h)) \to 0$.
We start by observing that
\begin{equation*}
\begin{aligned}
	0\le
	D_k
	&:=
	\dist\left ( G'(x) r_k, \TT_K^2(G(x), G'(x)h) - G''(x)h^2 \right  )
	\\&
	\leq \dist\left ( G'(x)r_k + G''(x)h^2 + \varphi_k,   \TT_K^2(G(x), G'(x)h)  \right  ) + \norm{\varphi_k}_Z
	\to 0,
\end{aligned}
\end{equation*}
where we used \eqref{eq:long_Taylor},
$t_kG'(x) r_k \to 0$
and the outer second-order regularity of $K$ at $G(x)$.

Now, let $\{\eta_k\} \subset Z$ be a sequence with
$G'(x)r_k + G''(x)h^2 + \eta_k \in \TT_K^2(G(x), G'(x)h)$
and $\norm{\eta_k}_Z \le D_k + 1/k$.
From \cite[Theorem~2.1]{ZoweKurcyusz1979} and \cite[Proposition 2.95]{BonnansShapiro}, we obtain that there exists a $\rho > 0$ with
\begin{equation*}
B_\rho^Z(0) \subseteq G'(x)B_1^X(0) - (K - G(x)) \cap B_1^Z(0).
\end{equation*}
In particular, we may find sequences $\mu_k \in X$ and $\lambda_k \in \RR_K(G(x))$ such that
\begin{equation}
\label{eq:randomequation321}
\eta_k = G'(x)\mu_k - \lambda_k\quad\text{and}\quad \|\mu_k\|_X \leq \rho^{-1} \|\eta_k\|_Z.
\end{equation}
From \eqref{eq:randomequation321} and the inclusions
\begin{subequations}
	\label{eq:inclusions_sos}
	\begin{align}
		\RR_K(G(x)) &\subseteq \TT_{\TT_K(G(x))}(h), \\
		\TT_K^2(G(x), G'(x)h) +  \TT_{\TT_K(G(x))}(G'(x)h) &\subseteq \TT_K^2(G(x), G'(x)h),
	\end{align}
\end{subequations}
which follow from the fact that $\TT_K(G(x))$ is a closed convex cone and \cite[Proposition 3.34]{BonnansShapiro},
we obtain
\begin{equation*}
G'(x)(r_k + \mu_k) \in  \TT_K^2(G(x), G'(x)h)  + \lambda_k -G''(x)h^2 \subseteq  \TT_K^2(G(x), G'(x)h)  -G''(x)h^2.
\end{equation*}
Using the identity
\begin{equation}
	\label{eq:pullback_sos}
 \TT_C^{2}( x,h) = G'(x)^{-1}\left (\TT_K^2(G(x), G'( x)h) - G''( x)(h,h) \right ) \quad \forall h \in  \TT_C( x)
\end{equation}
that is found, e.g., in \cite[Proposition 3.33]{BonnansShapiro}, we infer
\begin{equation*}
	r_k + \mu_k
	\in
	G'(x)^{-1}\left (\TT_K^2(G(x), G'(x)h) - G''(x)h^2 \right )
	=
	\TT_C^2(x, h)
	.
\end{equation*}
The above implies that we indeed have
\begin{equation*}
	\dist(r_k, \TT_C^2(x, h))
	\le
	\norm{\mu_k}_X
	\le
	\rho^{-1} \norm{\eta_k}_Z
	\le
	\rho^{-1} \, \bigh(){ D_k + 1/k} \to 0.
\end{equation*}
Arguing as in the first part of the proof of \Cref{lem:curvature_of_OSOR}, we now obtain
			\begin{equation*}
				 \liminf_{k \to \infty}\left \langle \varphi,  r_k \right \rangle   \geq \inf_{ r \in  \TT_C^{2}(x,h)}  \left \langle \varphi, r \right \rangle,
			\end{equation*}
and, as a consequence,
	\begin{equation*}
			 Q_C^{x,\varphi}( h )  \geq  \inf_{ r \in \TT_C^{2}(x,h)}  \left \langle \varphi, r \right \rangle.
	\end{equation*}
On the other hand, an argumentation analogous to that employed in the second part of the proof of  \Cref{lem:curvature_of_OSOR} yields
	\begin{equation*}
		\begin{aligned}
			 Q_C^{x,\varphi}( h )  
			&  \leq
			\inf\varh\{\}{
				\liminf_{k \to \infty}\left \langle \varphi,  r_k \right \rangle
				\mid
				\begin{aligned}
					\{r_k\} \subset X, \{t_k\} \subset \R^+ :{}
					&
					t_k \searrow 0,
					t_k \, r_k \to 0,
					\\ & \;
					x + t_k \, h + \frac12 \, t_k^2 \, r_k \in C
				\end{aligned}
			}~\\
			&\leq  \inf_{ r \in \TT_C^{2}(x,h)}  \left \langle \varphi, r \right \rangle
			\leq
			 Q_C^{x,\varphi}( h )  
			 .
		\end{aligned}
	\end{equation*}
	Hence, equality holds everywhere, \eqref{eq:pullbackformula} is valid (cf.\ \eqref{eq:pullback_sos}) and \eqref{eq:MRC} is satisfied in $(x,\varphi)$ by the observation in \Cref{remark_mrc}. This completes the proof.
 \end{proof}

Under stronger assumptions on $x$ and $G$,  the right-hand side of \eqref{eq:pullbackformula} is directly related to the directional curvature functional of $K$.

\begin{lemma}
	\label{lem:pull_back_unique_multiplier}
	In the situation of \Cref{lem:SORCurvatureTerm}, assume that $Z$ is the dual of a separable Banach space and that there exists a $\lambda \in \NN_K^\star(G(x))$ with $\varphi + G'(x)\adjoint \lambda = 0$  such that \eqref{eq:MRC} of $K$ holds in $(G(x), -\lambda)$ and
	\begin{equation}
		\label{eq:strong_robinson}
		Z = G'(x)X - \RR_K(G(x)) \cap \lambda\anni.
	\end{equation}
	Then, 
	\begin{equation}
	\label{eq:Lagrangepullbackformula}
		Q_C^{ x , \varphi}( h)
		=
		Q_K^{G( x), -\lambda }( G'( x) \, h)
		+
		\dual{\lambda}{ G''(x) \, h^2}\quad \forall h \in \KK_C^\star(x,\varphi).
	\end{equation}
\end{lemma}
\begin{proof}
	Let $h \in \KK_C^\star(x,\varphi)$ be arbitrary but fixed. From \eqref{eq:inclusions_sos}, \eqref{eq:strong_robinson} and \cite[Theorem~2.1]{ZoweKurcyusz1979}, it follows (analogously to the proof of \Cref{lem:SORCurvatureTerm}) that for every  $w \in \TT_K^2(G( x), G'(  x) \, h) - G''( x) \, h^2$ there exist an $r \in X$
	and an $\eta \in \RR_K(G( x)) \cap \lambda\anni$
	with
	\begin{equation*}
		w = G'(  x) \, r - \eta \quad \text{and}\quad
		G'( x) \, r
		=
		w + \eta
		\in
		\TT_K^2(G( x), G'( x) \, h) - G''( x) \, h^2.
	\end{equation*}
	Note that an $r$ with the latter properties necessarily satisfies $\dual{-\lambda}{w} = \dual{-\lambda}{G'(x)r}$ and $G'( x) \, r
		 \in
		G'(x)X \cap \left ( \TT_K^2(G( x), G'( x) \, h) - G''( x) \, h^2\right )$. Consequently, we may deduce
	\begin{equation}
		\label{eq:randomequality635}
		\inf_{w \in \TT_K^2(G( x), G'( x) \, h) - G''( x) h^2} \dual{-\lambda }{ w}
		\geq
		\inf_{w \in G'(x)X \cap \left ( \TT_K^2(G( x), G'( x) \, h) - G''( x) h^2 \right ) } \dual{-\lambda}{ w}.
	\end{equation}
	On the other hand, we trivially have 
	\begin{equation*}
		\inf_{w \in \TT_K^2(G( x), G'( x) \, h) - G''( x) h^2} \dual{-\lambda }{ w}
		\leq
		\inf_{w \in G'(x)X \cap \left ( \TT_K^2(G( x), G'( x) \, h) - G''( x) h^2 \right ) } \dual{-\lambda}{ w},
	\end{equation*}
	so  equality has to hold in \eqref{eq:randomequality635}.
	Using this equality, \eqref{eq:pullback_sos}, \Cref{lem:curvature_of_OSOR,lem:SORCurvatureTerm},
	the identity $\varphi + G'(x)\adjoint \lambda = 0$ and a straightforward calculation, we obtain 
	\begin{equation*}
	\begin{aligned}
		&Q_C^{ x , \varphi}( h)
		\\
		&\quad=
		\inf_{ r \in  \TT_C^2(x,h) }  \left \langle \varphi , r \right \rangle  
		=
		\inf_{ r \in  G'( x)^{-1}\left (\TT_K^2(G( x), G'( x)h) - G''( x)h^2 \right )}  \left \langle -\lambda , G'(x)r \right \rangle  
		\\ &\quad
		=
		\inf_{w \in G'(x)X \cap \left ( \TT_K^2(G( x), G'( x) \, h) - G''( x) h^2 \right ) } \dual{-\lambda}{ w}
		=
		\inf_{w \in \TT_K^2(G( x), G'( x) \, h) - G''( x) h^2} \dual{-\lambda }{ w}
		\\ &\quad
		=
		\inf_{ w \in  \TT_K^{2}(G(x), G'(x)h)}  \left \langle -\lambda, w \right \rangle
		+
		\dual{\lambda}{G''( x) \, h^2}
		=
		Q_K^{G( x), -\lambda}(G'( x)\,h) + \dual{\lambda}{G''( x) \, h^2}
		.
	\end{aligned}
	\end{equation*}
	This proves the claim. 
\end{proof}

Several things are noteworthy regarding \Cref{lem:pull_back_unique_multiplier} and its assumptions.

\begin{remark}\hfill
	\begin{enumerate}
	\item
		The pull-back formula \eqref{eq:Lagrangepullbackformula} is, in fact, valid
		in a setting that is far more general than the one considered in
		\Cref{lem:SORCurvatureTerm,lem:pull_back_unique_multiplier}. It holds,
		e.g., also for polyhedric sets $K$ provided the strengthened Zowe-Kurcyusz
		condition \eqref{eq:strong_robinson} is satisfied, cf.\ \cite[Theorem 5.7]{WachsmuthPolyhedricity}.
		It is further remarkable
		that the estimate
	\begin{equation*}
		Q_C^{ x , \varphi}( h)
		\geq
		Q_K^{G( x), -\lambda }( G'( x) \, h)
		+
		\dual{\lambda}{ G''(x) \, h^2}\quad \forall h \in \KK_C^\star(x,\varphi),
	\end{equation*}
which yields an SSC for the problem \eqref{eq:problem}, can often be proved
without any constraint qualifications at all. To avoid overloading this paper,
we leave a detailed discussion of the latter topics for future research. 

\item A possible interpretation of the formula \eqref{eq:Lagrangepullbackformula}   is that the (directional)
curvature of the set $C$ has its origin in the 
nonlinearity of $G$ or in the 
curvedness of $K$.

\item The condition  \eqref{eq:strong_robinson} is well-known and appears, e.g., also in the study of the uniqueness of Lagrange multipliers. It is precisely the ordinary Zowe-Kurcyusz constraint qualification for the set $\tilde{K} := \{ u \in K : (u - G(x)) \in \lambda\anni \}$. We refer to \cite[Theorem~2.2]{Shapiro1997} for details on this topic. 
\end{enumerate}
\end{remark}

\subsection{Two Corollaries of \texorpdfstring{\Cref{thm:no-gap-SOC}}{Theorem~\ref{thm:no-gap-SOC}} and Some Tangible Examples}
\label{subsec:PolyhedricitySORSecondOrderConditions}

If we combine the findings of \Cref{subsec:PolyhedricitySORCurvature,subsec:NDS} with the analysis of \Cref{sec:second-order_conditions}, then we arrive, e.g., at the following two results.

\begin{theorem}[No-Gap Second-Order Condition for Polyhedric Sets]
	\label{thm:no-gap-SOC_poly}
	Suppose that $X$ is reflexive and that $C$ is polyhedric at $\bar x$. 
	Assume that $J'(\bar x) \in -\NNs_C(\bar x)$ holds and that  $J''(\bar x)$ is a Legendre form in the sense of \Cref{lemma:sufficientfornondegeneracy} \ref{item:sufficientfornondegeneracy_2}.
	Then, the condition
	\begin{equation*}
		J''(\bar x) \, h^2
		>
		0
		\quad
		\forall h \in \KKs_C(\bar x, J'(\bar x)) \setminus \{0\}
	\end{equation*}
	is equivalent to the quadratic growth condition \eqref{eq:second_order_growth} with constants $c>0$ and $\varepsilon > 0$.
\end{theorem}

\begin{proof}
	Use  \Cref{lemma:sufficientfornondegeneracy,lemma:polyhedricity} in \Cref{thm:no-gap-SOC}.  
\end{proof}

\begin{theorem}[No-Gap Second-Order Condition under Second-Order Regularity]
\label{thm:no-gap-SOC_soreg}
Let $Z$ be a Banach space.
Assume that $C=G^{-1}(K)$ holds for some twice continuously Fréchet
differentiable function $G : X \to Z$ and
some closed, convex, non-empty set $K \subseteq Z$. Assume further that $J'(\bar x) \in -\NNs_C(\bar x)$ holds, that  $J''(\bar x)$ is  \mbox{weak-$\star$} lower semicontinuous, that \eqref{eq:NDS} holds, that the maps 
$h \mapsto G'(\bar x)h $ and  $h \mapsto G''(\bar x)h^2$ are  \mbox{weak-$\star$}-to-strong continuous, that $K$ is second-order regular in $G(\bar x)$, and that  the constraint qualification 
	\begin{equation*}
		G'( \bar x) \, X - \RR_K(G( \bar x)) = Z
	\end{equation*}
is satisfied. Then, the condition
	\begin{equation}
	\label{eq:pull_back_ssc}
	\begin{aligned}
		&J''(\bar x) \, h^2 +  \inf_{ r \in G'(\bar x)^{-1}\left (\TT_K^2(G(\bar x), G'(\bar x)h) - G''(\bar x)h^2\right )}  \left \langle J'(\bar x), r \right \rangle
		>
		0
		\\
		&\qquad\qquad\qquad\qquad\qquad\qquad \forall h \in J'(\bar x)^\perp \cap G'(\bar x)^{-1}\TT_K(G(\bar x)) \setminus \{0\}\,
	\end{aligned}
	\end{equation}
	is equivalent to the quadratic growth condition \eqref{eq:second_order_growth} with constants $c>0$ and $\varepsilon > 0$. If, moreover, we know that $Z$ is the dual of a separable Banach space and that there exists a Lagrange multiplier $\lambda \in \NN_K^\star(G(\bar x))$ satisfying $J'(\bar x)+ G'(\bar x)\adjoint \lambda = 0$  such that \eqref{eq:MRC} holds in $(G(\bar x), -\lambda)$ and such that
	\begin{equation*}
		Z = G'(\bar x)X - \RR_K(G(\bar x)) \cap \lambda\anni,
	\end{equation*}
then \eqref{eq:pull_back_ssc} is equivalent to 
	\begin{equation*}
	\begin{aligned}
		&\partial_{xx} L (x,\lambda)h^2
		+
		Q_K^{G(\bar x),-\lambda} (G'(\bar x)h)
		>
		0
		 \qquad \forall h \in J'(\bar x)^\perp \cap G'(\bar x)^{-1}\TT_K(G(\bar x)) \setminus \{0\},
	\end{aligned}
	\end{equation*}
where $L(x, \lambda ) := J(x) + \dual{\lambda}{G(x)}$ is the Lagrangian associated with \eqref{eq:problem}.
\end{theorem}
\begin{proof}
Use \Cref{lem:curvature_of_OSOR,lem:SORCurvatureTerm,lem:pull_back_unique_multiplier} in \Cref{thm:no-gap-SOC}.
\end{proof}

We remark that no-gap second-order conditions for specific classes of
optimization problems with polyhedric admissible sets (in particular, optimal
control problems with box constraints) can be found  frequently in the
literature. We only mention \cite[Theorem~2.7]{Bonnans1998},
\cite[Theorems~2.2, 2.3]{CasasTroeltzsch2012}, and
\cite[Theorem~4.13]{Casas2015}
as examples here.
\Cref{thm:no-gap-SOC_poly} reproduces these   results on an abstract level. 

Second-order conditions similar to those in \Cref{thm:no-gap-SOC_soreg}, on the other hand,  have been studied extensively in \cite{BonnansShapiro} in various formats and settings, see ibidem Theorems 3.45,  3.83, 3.86, 3.109, 3.137,  3.145, 3.148, 3.155 and Proposition 3.46. It should be noted that the no-gap conditions derived in \cite[Section 3.3.3]{BonnansShapiro} all require $X$ to be finite-dimensional and, in addition, all need further assumptions on, e.g., the second-order tangent set $\TT_K^2(G(\bar x), G'(\bar x)h)$. Such assumptions are not needed for the derivation of our second-order condition \eqref{eq:pull_back_ssc}, but may be required for reformulations of \eqref{eq:pull_back_ssc} as we have seen in the second part of \Cref{thm:no-gap-SOC_soreg}.

 We conclude this section with two simple examples that demonstrate the usefulness of  \Cref{thm:no-gap-SOC_poly,thm:no-gap-SOC_soreg}.

\begin{example}[A Simple Optimal Control Problem with Control Constraints]
\label{example:ControlConstraints}
Consider a minimization problem of the form 
\begin{equation}
	\label{eq:optctrlexample}
	\begin{aligned}
		\text{Minimize} \quad & j(y) + \frac{\gamma}{2} \int_\Omega u^2 \, \mathrm{d}\LL^d \\
		\text{such that} \quad & u \in L^2(\Omega),\quad -1 \leq u \leq 1 \text{ a.e.\ in }\Omega,\quad S(u) = y,
	\end{aligned}
\end{equation}
where $\Omega \subseteq \mathbb{R}^d$ is a bounded domain, $j: L^\infty(\Omega) \to \mathbb{R}$ is twice continuously differentiable,
$\gamma > 0$ is a Tikhonov parameter, $\LL^d$ is the Lebesgue measure,
and $S : L^2(\Omega) \to L^\infty(\Omega)$ is  (for simplicity) linear and compact.
In this situation, the space $X := L^2(\Omega)$ is obviously reflexive,
the set $C := L^2(\Omega, [-1,1])$ is closed, non-empty and polyhedric at every point,
and
the reduced objective $J(u) := j(Su) + \frac{\gamma}{2} \int_\Omega u^2 \, \mathrm{d}\LL^d$ is a $C^2$-function with 
\begin{equation*}
\begin{aligned}
J'(u)h &= j'(Su)(Sh) + \gamma (u, h)_{L^2(\Omega)} && \forall u, h \in L^2(\Omega), \\
J''(u)(h_1, h_2) &= j''(Su)(Sh_1, Sh_2)  + \gamma (h_1,h_2)_{L^2(\Omega)} && \forall u, h_1, h_2 \in L^2(\Omega).
\end{aligned}
\end{equation*}
Note that the map $h \mapsto J''(u) \, h^2$ is weakly lower semicontinuous for all $u \in L^2(\Omega)$, and that $J''(u)$ is a Legendre form for all $u \in L^2(\Omega)$ since 
\begin{equation}
\label{eq:L2Legendre}
\begin{aligned}
 h_k \weakly  h \text{ and } J''(u) \, h_k^2 \to J''(u) \, h^2  
 &\quad \Rightarrow \quad h_k \weakly  h \text{ and } \|h_k\|_{L^2(\Omega)}^2 \to \|h\|_{L^2(\Omega)}^2 
 \\
 &\quad \Rightarrow \quad h_k \to h.
\end{aligned}
\end{equation}
Consequently, \Cref{thm:no-gap-SOC_poly} is applicable in case of  problem \eqref{eq:optctrlexample},
and we may deduce that,
given some $\bar u \in C$
with $\bar p + \gamma \bar u \in -\NNs_C(\bar u)$,
where $\bar p = S^\star (j'(S\bar u))$ is the adjoint state,
the condition 
	\begin{equation*}
		 j''(S\bar u)(Sh, Sh)  + \gamma \|h\|_{L^2(\Omega)}^2
		>
		0
		\ \ 
		\forall h \in \KKs_C(\bar u, \bar p + \gamma \bar u) \setminus \{0\}
	\end{equation*}
is equivalent to the quadratic growth condition 
	\begin{equation*}
		J(u)
		\ge
		J(\bar u) + \frac{c}{2} \, \norm{u- \bar u}^2_{L^2(\Omega)}
		\qquad\forall u \in C \cap  B_\varepsilon^{L^2(\Omega)}(\bar u)
	\end{equation*}
	 with $c > 0$ and some $\varepsilon > 0$.
	 Note that the Tikhonov regularization is of particular importance in the above setting:
	 The condition $\gamma>0$ ensures (in combination with the compactness of $S$) that the derivative $J''(u)$ is a Legendre form for all $u \in L^2(\Omega)$ and thus guarantees \eqref{eq:NDS}.
	 We will see in \Cref{subsec:bangbang} that the situation changes drastically when the regularization parameter $\gamma$ equals zero.
\end{example}

\begin{example}[A Simple Optimal Control Problem with a Scalar Constraint]
\label{ex:state_comstraint}
Consider a minimization problem of the form 
\begin{equation*}
	\begin{aligned}
		\text{Minimize} \quad & j(y) + \frac{\gamma}{2} \int_\Omega u^2 \, \mathrm{d}\LL^d \\
		\text{such that} \quad & u \in L^2(\Omega),\quad Su = y,\quad Tu \in B_1^H(0),
	\end{aligned}
\end{equation*}
where $\Omega \subseteq \mathbb{R}^d$ is a bounded domain,
$j: L^\infty(\Omega) \to \mathbb{R}$ is twice continuously differentiable,
$\gamma > 0$ is a Tikhonov parameter,
$\LL^d$ is the Lebesgue measure,
$H$ is some Hilbert space,
and $S : L^2(\Omega) \to L^\infty(\Omega)$, $T : L^2(\Omega) \to H$ are (for simplicity) linear and compact.
Define $X := L^2(\Omega)$, $Z := \mathbb{R}$,
$G(u):= \|Tu\|_H^2 - 1$, $K:=(-\infty, 0]$, $C := G^{-1}(K)$,
and $J(u) := j(Su) + \frac{\gamma}{2} \int_\Omega u^2 \, \mathrm{d}\LL^d$.
Then $X$ and $Z$ are Hilbert spaces, $C$ is non-empty, convex and closed, $K$ is non-empty, convex and closed,
$G$ is twice continuously differentiable with $G'(u)h = 2(Tu, Th)_H$ and $G''(u)(h_1, h_2) = 2(Th_1, Th_2)_H$ for all $u \in L^2(\Omega)$,
$J$ is twice continuously differentiable with the same derivatives as in \Cref{example:ControlConstraints},
the map $h \mapsto J''(u) \, h^2$ is weakly lower semicontinuous for all $u \in L^2(\Omega)$,
$J''(u)$ is a Legendre form for all $u \in L^2(\Omega)$ (cf.\  \Cref{example:ControlConstraints}),
\eqref{eq:ZKCQ} is satisfied in every $u \in C$ (just use a distinction of cases),
and $K$ is second-order regular at every $z \in K$ with 
\begin{equation*}
\TT_K(z) =
\begin{cases}
\mathbb{R} & \text{if } z \in (-\infty,0), \\
(-\infty, 0]& \text{if } z = 0,
\end{cases}
\quad\!
\TT_K^2(z, h) = 
\begin{cases}
\mathbb{R} & \text{if } z \in (-\infty, 0), \ h \in \mathbb{R}, \\
\mathbb{R} & \text{if } z = 0, h \in (-\infty, 0),\\
(-\infty, 0] & \text{if } z=0, h=0.
\end{cases}
\end{equation*}
Now, let $\bar u \in C$ be given, such that
$\bar p + \gamma \bar u \in -\NNs_C(\bar u)$,
where $\bar p = S^\star (j'(S\bar u))$ is the adjoint state.
Using the above observations and \eqref{eq:pull_back_ssc}, we obtain that
the condition 
	\begin{equation}
	\label{eq:quadraticgrowthExample2}
	\begin{aligned}
		&j''(S\bar u)(Sh, Sh)  + \gamma \|h\|_{L^2(\Omega)}^2  +  \inf_{ r \in  G'(\bar u)^{-1}\left (\TT_K^2(G(\bar u), G'(\bar u)h) - 2\|Th\|_H^2 \right )}  \left \langle \bar p + \gamma \bar u, r \right \rangle  
		>
		0
		\\
		&\qquad\qquad\qquad\qquad\qquad\qquad \forall h \in \bigh(){ \bar p + \gamma \bar u }^\perp \cap G'(\bar u)^{-1}\TT_K(G(\bar u)) \setminus \{0\}
	\end{aligned}
	\end{equation}
is equivalent to the quadratic growth condition \eqref{eq:second_order_growth} with constants $c>0$ and $\varepsilon > 0$.
Note that $G'(\bar u) : L^2(\Omega) \to \mathbb{R}$ is surjective if $\|T \bar u \|_H > 0$.
Consequently, if $0 < \|T\bar u\| \leq1$,
we may use the second part of \Cref{thm:no-gap-SOC_soreg} to simplify  \eqref{eq:quadraticgrowthExample2}.
This yields
	\begin{equation*}
	\begin{aligned}
		&j''(S\bar u)(Sh, Sh)  + \gamma \|h\|_{L^2(\Omega)}^2  
		>
		0
		\quad \forall h \in \bigh(){ \bar p + \gamma \bar u }^\perp  \setminus \{0\}
	\end{aligned}
	\end{equation*}
for the case $0 < \|T\bar u\| < 1$ and the condition 
	\begin{equation*}
	\begin{aligned}
		&j''(S\bar u)(Sh, Sh)  + \gamma \|h\|_{L^2(\Omega)}^2  +   2\lambda\|Th\|_H^2  
		>
		0
		\\
		&\qquad\qquad\qquad\qquad \forall  h \in \bigh(){ \bar p + \gamma \bar u }^\perp \cap \left \{ h \in L^2(\Omega) \setminus \{0\} :  (Tu, Th)_H \leq 0 \right \}
	\end{aligned}
	\end{equation*}
for the case $\| T\bar u\|_H = 1$.  Here, $\lambda \geq 0$ is the (in this case necessarily unique)  Lagrange multiplier associated with $\bar u$. 
\end{example}

We remark that \Cref{ex:state_comstraint} can also be studied with different means. We chose the approach with the second-order
regularity here to illustrate \Cref{thm:no-gap-SOC_soreg}.

\section{Advantages of our Approach}
\label{sec:advantages}

Having demonstrated that the framework of  \Cref{sec:second-order_conditions} indeed allows to reproduce classical
results for minimization problems with polyhedric and second-order regular
sets, we now turn our attention to the benefits offered by our approach
in comparison with the classical theory.
The main advantages of our method are the following.

\begin{enumerate}
	\item Our approach splits the task of proving no-gap second-order optimality conditions for problems of the type \eqref{eq:problem}
		into subproblems that can be tackled independently from each other (namely, verifying  \eqref{eq:NDS}, checking the differentiability of $J$, and computing the directional curvature functional $Q_C^{ x , \varphi}(\cdot)$). We can further state our second-order conditions without imposing any preliminary assumptions (as, e.g., polyhedricity or second-order regularity) on the admissible 
		set $C$, cf.\ \Cref{thm:no-gap-SOC}. All of this makes our method more
		flexible than the classical ``all-at-once'' approach.
	\item Our results can also be employed in situations where the admissible set
		exhibits a singular or degenerate curvature behavior, cf.\ the examples in \Cref{subsec:singularcurvature,subsec:bangbang}.
	\item Our approach does not require a detailed analysis of the curvature of the set
		$C$. To obtain the second-order condition in \Cref{thm:no-gap-SOC}, we only
		have to study the behavior of the quantity
		$\left \langle J'(\bar x), r_k \right \rangle $
		that appears in the definition of the 
		functional $\smash{Q_C^{\bar x,J'(\bar x)}(\cdot)}$, i.e., we only have to analyze how the derivative
		$J'(\bar x)$ acts on the second-order corrections $r_k$ and not how the $r_k$
		behave in detail. This is a major difference to the concept of second-order
		regularity, cf.\ \Cref{def:polyhedric_SOCT}  \ref{def:polyhedric_SOCT:iii},
		and often very advantageous since it allows to exploit additional
		information about the gradient of the objective. We will
		see this effect in \Cref{subsec:bangbang} below.
\end{enumerate}

In the following, we demonstrate by means of two tangible examples that the above points are not only of academic interest but also of relevance in practice.
We begin with a simple finite-dimensional optimization problem whose admissible set exhibits a singular curvature behavior.

\subsection{Singular Curvature in Finite Dimensions}
\label{subsec:singularcurvature}

Consider a two-dimensional optimization problem of the form 
\begin{equation}
\label{eq:singular_curvature_example}
	\begin{aligned}
		\text{Minimize} \quad & J(x), &
		\text{such that} \quad & x \in C = \{(x_1, x_2) \in \R^2 \mid x_2 \ge \abs{x_1}^{\alpha}\}
	\end{aligned}
\end{equation}
with a twice continuously differentiable objective  $J : \R^2 \to \R$ and some $\alpha \in (1,2)$. Set $\bar x := (0,0)$ and suppose that $\bar x$ is a critical point of  \eqref{eq:singular_curvature_example} with a non-vanishing gradient, i.e., $  J'(\bar x) = (0, \beta) \in - \NN_C^\star(\bar x) = \{0\} \times[0, \infty)$ for some $\beta > 0$. 
Then, for every critical direction  $h \in \KK_C^\star (\bar x, J'(\bar x)) = \{(h_1,h_2) \in \R^2 \mid h_2 = 0\}$ and all sequences  $\{t_k\} \subset \R^+$ and $\{r_k\} \subset \R^2$
satisfying $t_k \searrow 0$, $t_k \, r_k \to 0$ and $\bar x + t_k \, h + \frac12 \, t_k^2 \, r_k \in C$, it holds
\begin{equation}
\label{eq:randomlenghtyestimate42}
\begin{aligned}
	&\liminf_{k \to \infty} \dual{J'(\bar x)}{r_k}  
	\\
	&\qquad
	= 2 \,
	\liminf_{k \to \infty} J'(\bar x)^\top \frac{\bar x + t_k \, h + \frac12 \, t_k^2 \, r_k}{t_k^2 }
	= 2\beta\,
	\liminf_{k \to \infty} \frac{(\bar x + t_k \, h + \frac12 \, t_k^2 \, r_k)_2}{t_k^2 }
	\\
	&\qquad
	\ge 2\beta\,
	\liminf_{k \to \infty} \frac{\abs{(\bar x + t_k \, h + \frac12 \, t_k^2 \, r_k)_1}^{\alpha}}{t_k^2 }
	= 2\beta\,
	\liminf_{k \to \infty} \abs{( h + \frac12 \, t_k \, r_k)_1}^{\alpha}   t_k^{\alpha - 2}.
\end{aligned}
\end{equation}
The above implies 
\begin{equation}
\label{eq:random_curvature_functional}
	Q_C^{\bar x, J'(\bar x)}(h)
	=
	+ \infty
	\qquad
	\forall  h \in \KK_C^\star(\bar x, J'(\bar x)) \setminus \{0\} = \{h \in \R^2 \mid h_1 \neq 0, \, h_2 = 0\}.
\end{equation}
Using \eqref{eq:random_curvature_functional}, \Cref{thm:no-gap-SOC} and the fact that the conditions \eqref{eq:MRC} and \eqref{eq:NDS} are trivially satisfied for \eqref{eq:singular_curvature_example}, we obtain (analogous to \cite[Example 3.84]{BonnansShapiro})
the following result.

\begin{theorem}
\label{th:singular_curvature_example_result}
If $\bar x = (0,0)$ is a critical point of \eqref{eq:singular_curvature_example}  with  $J'(\bar x) \neq 0$, then $\bar x$ is a local minimizer of \eqref{eq:singular_curvature_example} and there exist parameters $c>0$ and $\varepsilon > 0$ such that the quadratic growth condition \eqref{eq:second_order_growth} is satisfied. 
\end{theorem}

Note that the second derivative $J''(\bar x)$ is not important for local optimality of $\bar x$.
The reason for this is that the curvature of the boundary $\partial C$ is singular at the origin and can thus  compensate for any negative curvature that the objective $J$ might have at $\bar x$. It should be noted further that the set $C$ in  \eqref{eq:singular_curvature_example} is neither polyhedric (trivially) nor second-order regular at $\bar x$ (since $\TT_C^2(\bar x, h) = \emptyset$ for all $h \in  \KK_C^\star (\bar x, J'(\bar x)) \setminus \{0\}$, cf.\ \Cref{lem:nec_cond_osor},  and the estimate \eqref{eq:randomlenghtyestimate42}). This demonstrates that \eqref{eq:singular_curvature_example} does not fall under the setting of \Cref{thm:no-gap-SOC_poly,thm:no-gap-SOC_soreg} and is indeed not covered by what is typically seen as the classical second-order theory.

 We remark that, given an optimization problem of the form 
\begin{equation}
\label{eq:singular_curvature_example2}
	\begin{aligned}
		\text{Minimize} \quad & J(x) , &
		\text{such that} \quad & x \in C = \{(x_1, x_2) \in \R^2 \mid x_2 \le \abs{x_1}^{\alpha}\}
	\end{aligned}
\end{equation}
with $J \in C^2(\R^2)$, $J'(\bar x) = (0, - \beta) \in - \NN_C^\star(\bar x) = \{0\} \times(- \infty, 0]$, $\beta > 0$, $\bar x := (0,0)$, we can use exactly the same arguments as for \eqref{eq:singular_curvature_example} to prove
\begin{equation*}
	Q_C^{\bar x, J'(\bar x)}(h)
	=
	%
	- \infty
	\qquad \forall h \in \KK_C^\star(\bar x, J'(\bar x)) = \{h \in \R^2 \mid h_2 = 0\}.
\end{equation*}
The above yields in combination with \Cref{thm:FONC,thm:SNC}
that $\bar x = (0,0)$ can never be a local minimizer of \eqref{eq:singular_curvature_example2} unless the derivative $J'(\bar x)$ is identical zero.


\subsection{No-Gap Second-Order Conditions for Bang-Bang Problems}
\label{subsec:bangbang}
In this section,
we
demonstrate that the analysis of \Cref{sec:second-order_conditions} is not only
relevant for finite-dimensional toy problems à la
\eqref{eq:singular_curvature_example} and \eqref{eq:singular_curvature_example2},
but also applicable in more complicated situations.
In what follows, we will use it to derive no-gap second-order conditions for bang-bang optimal control problems.
As a motivation, let us consider the optimization problem \eqref{eq:optctrlexample} in \Cref{example:ControlConstraints} with $\gamma = 0$, i.e., the problem
\begin{equation}
\label{eq:bangbangconcrete}
	\begin{aligned}
		\text{Minimize} \quad &  j(y)  \\
		\text{such that} \quad & u \in L^2(\Omega),\quad -1 \leq u \leq 1 \text{ a.e.\ in }\Omega,\quad S(u) = y.
	\end{aligned}
\end{equation}
From \Cref{thm:FONC}, we obtain that every minimizer $\bar u$ of \eqref{eq:bangbangconcrete} satisfies  $\bar p  \in -\NNs_C(\bar u)$,
where $\bar p = S^\star (j'(S\bar u))$ is the adjoint state.
In particular, this implies that a minimizer $\bar u$ with $\LL^d(\{ \bar p = 0\})= 0$ can only take the values 
$\pm 1$ a.e.\ in $\Omega$.
Such a solution $\bar u$ is called bang-bang.

The major problem that arises when SSC for a bang-bang control $\bar u$ are considered is the verification of the non-degeneracy condition \eqref{eq:NDS}.
Recall that in \Cref{example:ControlConstraints} the latter is satisfied since the Tikhonov regularization causes the second derivative of the reduced objective to be a Legendre form in $L^2(\Omega)$, cf.\ \eqref{eq:L2Legendre}. For \eqref{eq:bangbangconcrete} such an argumentation is obviously not possible and this is not an artificial problem: It can be shown that 
quadratic growth in $L^2(\Omega)$ is in general not possible for a bang-bang solution $\bar u$ of \eqref{eq:bangbangconcrete}, i.e., the growth condition \eqref{eq:second_order_growth}
typically does not hold with $X = L^2(\Omega)$ and $c, \varepsilon > 0$,
cf.\ \cite[end of Section~2]{Casas2012:1}.

Hence, \Cref{thm:no-gap-SOC} cannot be applicable when we work with
the space $X = L^2(\Omega)$. Note that, if we calculate the critical cone
$\KKs_C(\bar u, \bar p)$ for a bang-bang control $\bar u$ in the $L^2$-setting,
then we end up with
$\KKs_C(\bar u, \bar p) = \TT_C(\bar u ) \cap \bar p \anni = \{0\}$
so that the conditions
$j''(S\bar u)(Sh, Sh) >  0$  $\forall h \in \KKs_C(\bar u, \bar p) \setminus \{0\}$ and
$j''(S\bar u)(Sh, Sh) \geq  0$ $\forall h \in \KKs_C(\bar u, \bar p)$,
which are the natural candidates for the
SSC and SNC, respectively, are both void.
This also indicates that it is not useful to discuss
\eqref{eq:bangbangconcrete} as a problem in $X = L^2(\Omega)$. 

The above discussion shows that, if we want to derive no-gap second-order conditions for a bang-bang optimal control problem of the type \eqref{eq:bangbangconcrete}, then we have to work with a space $X$ that is different from $L^2(\Omega)$.
As it turns out, the right choice is the measure space $X = \MM(\Omega)$.
Therefore, we consider the following setting.
\begin{assumption}[Standing Assumptions and Notation for the Bang-Bang Setting]
\label{assumption:bangbangsetting} 
We suppose that
$\Omega \subset \R^d$, $d \ge 1$, is a bounded domain with a Lipschitz boundary,
cf.\ \cite[Definition~1.2.1.1]{Grisvard1985}.
We define the space $Y := C_0(\Omega) =  \mathrm{cl}_{\|.\|_\infty} \left ( C_c(\Omega) \right )$ endowed with the usual supremum norm.
Its dual space can be identified with
$X := \MM(\Omega)$ which is the space of signed finite Radon measures on $\Omega$
endowed with the norm $\|\mu\|_{\MM(\Omega)}:=|\mu|(\Omega)$, cf.\ \cite[Theorem 1.54]{Ambrosio}.
The space $L^1(\Omega)$ is identified with a closed subspace of $\MM(\Omega)$ via the isometric embedding $x \mapsto x \LL^d$, where $\LL^d$ is Lebesgue's measure.
Note that this implies $\|x\|_{L^1(\Omega)} = \|x\|_{\MM(\Omega)}$ for all $x \in L^1(\Omega)$.
Finally,
$ C :=  L^\infty(\Omega, [-1,1]) = \{x \in L^\infty(\Omega) : -1 \le x \le 1 \text{ a.e.\ in } \Omega\} \subseteq \MM(\Omega)$.
\end{assumption}

The reason for using the space $X = \MM(\Omega)$ is the following observation, cf.\ \cite[Proposition~2.7]{CasasWachsmuthsBangBang}.

\begin{lemma}
\label{lemma:bangbangNDS}
Let $\bar x \in C$ and $\bar\varphi \in -\NN_C^\star(\bar x)$ be given. Define
\begin{equation}
	\label{eq:measure_of_set}
	K(\bar \varphi)
	:=
	\frac{1}{4} 
	\liminf_{s \searrow 0} \left ( \frac{s}{\LL^d(\{ |\bar\varphi| \leq s\})} \right )
	\in
	[0,+\infty].
\end{equation}
Then, there exists a family of constants $\{c_\varepsilon\}$ satisfying $c_\varepsilon \searrow 0$ as $\varepsilon \searrow 0$ such that
\begin{equation}
\label{eq:bangbangNDCestimate}
\dual{\bar \varphi}{x-\bar x}_{C_0(\Omega),\MM(\Omega)}   \ge \left ( \frac{1}{2} K(\bar \varphi) - c_\varepsilon  \right ) \norm{x - \bar x}_{L^1(\Omega)}^2\quad \forall x \in C \cap B_\varepsilon^X(\bar x)\quad \forall \varepsilon > 0.
\end{equation}
\end{lemma}
\begin{proof}
We adapt the proof of \cite[Proposition~2.7]{CasasWachsmuthsBangBang}.
Define $K := K(\bar \varphi)$. 
If $K=0$, then the claim is trivially true. If $K > 0$, then it necessarily holds that $\LL^d(\{ \bar\varphi = 0\}) = 0$
and we obtain from $\bar \varphi \in -\NN_C^\star(\bar x) $
that $\bar x$ is bang-bang with $\bar x = - \mathrm{sgn}\, \bar \varphi$ a.e.\ in $\Omega$.
Using the latter and $\NN_C^\star(\bar x) \subseteq C_0(\Omega)$,
we may calculate that for all $h \in X$ and all $t>0$ with $\bar x + th \in C$ and $\|h\|_X = \|h\|_{L^1(\Omega)} = 1$, it holds
\begin{equation*}
\begin{aligned}
	\int_\Omega  \bar \varphi \frac{ h}{t} \mathrm{d}\LL^d 
	&=   \int_\Omega  |\bar \varphi| \frac{ | h  |}{t } \mathrm{d}\LL^d 
	\geq \int_{\{|\bar \varphi| > Kt\}}  |\bar \varphi| \frac{ | h |}{t } \mathrm{d}\LL^d 
	\\
	&\geq \int_{\Omega}    K    | h  |  \mathrm{d}\LL^d -  \int_{\{| \bar \varphi | \leq K  t  \}}  K   | h |  \mathrm{d}\LL^d
	\geq K -  K \|h\|_{L^\infty(\Omega)} \LL^d(\{ |\bar\varphi| \leq K t \}).
\end{aligned}
\end{equation*}
By using $t := \|x - \bar x\|_{L^1(\Omega)}$, $h := (x - \bar x)/\|x - \bar x\|_{L^1(\Omega)}$ and $\|h\|_{L^\infty(\Omega)} \leq 2/t$,
this implies
\begin{equation*}
\dual{\bar \varphi}{x-\bar x}_{C_0(\Omega),\MM(\Omega)}   \ge \left (  K - 2K^2 \sup_{0 <t \leq \varepsilon} \frac{\LL^d(\{ |\bar\varphi| \leq Kt \})}{Kt}  \right ) \norm{x - \bar x}_{L^1(\Omega)}^2
\end{equation*}
for all $x \in C \cap B_\varepsilon^X(\bar x) \setminus \{\bar x \}$ and all $\varepsilon > 0$.
Note that the coefficient on the right-hand side of the last estimate satisfies
\begin{equation*}
\lim_{\varepsilon \searrow 0}  \left (  K - 2K^2 \sup_{0 < t \leq \varepsilon} \frac{\LL^d(\{ |\bar\varphi| \leq Kt \})}{Kt}  \right )
	=   K - 2K^2 \left (  \limsup_{t \searrow 0}  \frac{\LL^d(\{ |\bar\varphi| \leq t\})}{t} \right ) = \frac{1}{2}K.
\end{equation*}
This proves the claim.
\end{proof}

\Cref{lemma:bangbangNDS} shows that  \Cref{lemma:sufficientfornondegeneracy}~\ref{item:sufficientfornondegeneracy_3} is applicable when we consider a bang-bang solution $\bar x$ whose gradient $\bar \varphi := J'(\bar x)$ satisfies $K(\bar \varphi) > 0$. This allows us to verify \eqref{eq:NDS} and to obtain the following result from \Cref{thm:no-gap-SOC}.

\begin{theorem}[No-Gap Second-Order Condition for Bang-Bang Problems]
	\label{thm:no-gap-measures}
	We consider an optimization problem of the form \eqref{eq:problem}
	with $C$, $X$ etc.\  as in \Cref{assumption:bangbangsetting}. Assume that $\bar x \in C$ is fixed, that $J$ satisfies the conditions in  \Cref{asm:standing_SOC}, that the map
	$ X \ni h \mapsto J''(\bar x) \, h^2 \in \R$ is  \mbox{weak-$\star$} continuous, that $\bar \varphi := J'(\bar x) \in -\NNs_C(\bar x)$, and that the constant $K(\bar \varphi)$ in \eqref{eq:measure_of_set} is positive. 
	Then, the condition
	\begin{equation}
		\label{eq:soc_measures}
		Q_C^{\bar x, \bar \varphi }(h)  + J''(\bar x) \, h^2
		>
		0
		\quad
		\forall h \in \KKs_C(\bar x, \bar \varphi) \setminus \{0\}
	\end{equation}
	is equivalent to the quadratic growth condition
	\begin{equation}
	\label{eq:bangbanggrowth}
		J(x) \ge J(\bar x) + \frac{c}{2} \, \norm{ x - \bar x }_{L^1(\Omega)}^2\quad\forall x \in C \cap B_\varepsilon^X(\bar x)
	\end{equation}
	with constants $c>0$ and $\varepsilon > 0$.
\end{theorem}
\begin{proof}
	We only need to check that \Cref{thm:no-gap-SOC} is applicable.
	The setting in \Cref{assumption:bangbangsetting} clearly fits into that of \Cref{asm:standing_assumption}, and 
	\Cref{asm:standing_SOC} is trivially satisfied.
	From $K(\bar \varphi) > 0$ and \Cref{lemma:bangbangNDS}, we further obtain that there exist 			constants $c, \varepsilon > 0$ with $J'(\bar x) \, (x - \bar x) \ge \frac{c}{2} \, \norm{x - \bar x}^2$ for all $x \in B^X_\varepsilon(\bar x) \cap C$. This yields, in combination with the  \mbox{weak-$\star$} continuity of $h \mapsto J''(\bar x)h^2$ and	\Cref{lemma:sufficientfornondegeneracy}~\ref{item:sufficientfornondegeneracy_3}, that \eqref{eq:NDS} holds. Note that the  \mbox{weak-$\star$} continuity of $h \mapsto J''(\bar x)h^2$ also implies  \ref{assumption-usc} in \Cref{thm:SNC}. This shows that \Cref{thm:no-gap-SOC} is applicable and proves the claim.
\end{proof}

Some remarks concerning \Cref{lemma:bangbangNDS,thm:no-gap-measures} are in order.
\begin{remark}~
	\label{rem:no-gap_measure}
	\begin{enumerate}
		\item
			\Cref{thm:no-gap-measures} provides no-gap second-order conditions
			even in the case that we cannot characterize the directional curvature functional
			$\smash{Q_C^{\bar x,\bar \varphi }(\cdot)}$ precisely.
		\item
			Recall that the condition in \Cref{lemma:sufficientfornondegeneracy}~\ref{item:sufficientfornondegeneracy_3} not only implies \eqref{eq:NDS} but also yields a coercivity estimate for the functional $\smash{Q_C^{\bar x, \bar \varphi}(\cdot)}$, see \eqref{eq:coercivecurvature}. Using this estimate and \eqref{eq:bangbangNDCestimate}, we obtain that, in the situation of \Cref{thm:no-gap-measures},
\begin{equation}
\label{eq:bangbangcoercivity}
	Q_C^{\bar x, \bar \varphi}(h) \ge K( \bar \varphi  ) \, \norm{h}_X^2
	\qquad\forall h \in \KKs_C(\bar x, \bar \varphi).
\end{equation}
Here, $K(\bar \varphi)>0$ is again defined by \eqref{eq:measure_of_set}.
We point out that \eqref{eq:bangbangcoercivity} implies that the set $C = L^\infty(\Omega, [-1,1])$ possesses positive curvature as a subset of the space $\MM(\Omega)$.
This is not true if $C$ is considered as a subset of the space $L^2(\Omega)$ as we have seen in \Cref{example:ControlConstraints} (in $L^2(\Omega)$,
$C$ is polyhedric and the curvature functional is zero).
			
		\item
			\label{rem:no-gap_measure_iii}
			From \eqref{eq:bangbangcoercivity}, it follows that $ J''(\bar x) \, h^2 > - K(\bar \varphi)\, \norm{h}^2_X$
			for all $h \in \KKs_C(\bar x, \bar \varphi ) \setminus \{0\}$ is a sufficient condition for quadratic growth in the situation of \Cref{thm:no-gap-measures}, cf.\  \eqref{eq:soc_measures}.
			We point out that this SSC is sharper than that found in \cite[Corollary~2.15]{CasasWachsmuthsBangBang}.
			In this contribution the authors work with the slightly more restrictive 
			``global'' level set assumption
			\begin{equation*}
				\tilde K  \leq \frac{s}{4 \, \LL^d(\{ |\bar\varphi| \leq s\})} \quad \forall s>0
			\end{equation*}
			for some $\tilde K > 0$.
			Note that such a $\tilde K$ necessarily satisfies $\tilde K \leq K(\bar\varphi)$.
			We further point out that the SSC in \cite[Corollary~2.15]{CasasWachsmuthsBangBang}
			can be improved to
			\begin{equation*}
				\exists \varepsilon > 0 : \quad J''(\bar x) \, h^2 \geq - \bigh(){ \tilde K  - \varepsilon } \norm{h}^2_X\quad  \forall h \in \KKs_C(\bar x, \bar \varphi ),
			\end{equation*}
			cf.\ \cite[Theorem~2.4]{CasasWachsmuthsBangBang2}.

		\item
			We expect that the SSC $ J''(\bar x) \, h^2 > - K(\bar \varphi)\, \norm{h}^2_X$
			$ \forall h \in \KKs_C(\bar x, \bar \varphi ) \setminus \{0\}$ can also be formulated as an inequality on the so-called extended critical cone introduced in \cite{Casas2012:1}, cf.\ \cite[Theorem~2.14]{CasasWachsmuthsBangBang} and
			\cite[Theorem~2.4]{CasasWachsmuthsBangBang2}.   We do not pursue this approach here. 
	\end{enumerate}
\end{remark}

The next step is the calculation of
the curvature functional
$\smash{Q_C^{\bar x, \bar \varphi}(\cdot)}$ for a bang-bang solution $\bar x$
in order
to obtain a no-gap optimality condition that is more explicit than \eqref{eq:soc_measures}.
Hence, we have to compute the (directional) curvature of the set $C :=  L^\infty(\Omega, [-1,1])$ as a subset of the space $\MM(\Omega)$.
In the remainder of this section,
we will consider a bang-bang solution $\bar x$
whose gradient $\bar \varphi := J'(\bar x) \in -\NNs_C(\bar x)$ is in $C^1(\Omega)$.
Let us first fix our assumptions on the  $\bar x$ under consideration.

\begin{assumption}[Assumptions and Notation for the Calculation of \texorpdfstring{$\smash{Q_C^{\bar x, \bar \varphi}(\cdot)}$}{Q}]\hfill\par
\label{assumption:bangbangsolution} 
In addition to \Cref{assumption:bangbangsetting},
we suppose that 
$\bar x \in C$ and $\bar \varphi \in - \NN_C^{\star}(\bar x)$ are given.
We require
$\bar \varphi \in \iota(C_0(\Omega) \cap C^1(\Omega))$
and define
$\ZZ := \{z \in \Omega : \bar \varphi(z) = 0\}$.
We assume $\ZZ \subset \{z \in \Omega : \abs{\nabla \bar \varphi (z)} \neq 0 \}$.
Here and in the sequel, $\abs{\nabla \bar \varphi(z)}$
denotes the Euclidean norm of $\nabla \bar\varphi(z) \in \R^d$.

Finally, we denote by
$\HH^{d-1}$ the $(d-1)$-dimensional Hausdorff measure, which is scaled as in \cite[Definition 2.1]{EvansGariepy}.
\end{assumption}

In the above situation, the set $\ZZ$ is a $(d-1)$-dimensional $C^1$-submanifold of $\R^d$ due to the implicit function theorem, cf.\ \cite[Theorem 2.32]{Holm}.
This implies in particular that $\LL^d(\ZZ) = \LL^d(\{\bar \varphi = 0\}) = 0$ and that $\bar x$ is indeed bang-bang with $\bar x = - \mathrm{sgn}\, \bar \varphi$ a.e.\ in $\Omega$.
To calculate $\smash{Q_C^{\bar x, \bar \varphi}(\cdot)}$, we need the following
directional Taylor-like expansion of the $L^1(\Omega)$-norm.

\begin{lemma}[{\texorpdfstring{\cite[Corollary 5.10]{ChristofMeyerVI}}{[Corollary~5.10; Christof, Meyer 2016]}}]
\label{lemma:absoluteTaylorexpansion}
Given \Cref{assumption:bangbangsolution}, 
for all $v \in C_c (\Omega) \cap H^1(\Omega)$ and all sequences $t_k \in (0, \infty)$ with $t_k \searrow 0$, it is true that
\begin{equation}
\label{eq:nonstandardtaylorexpansion}
 \int_{\Omega}  \abs{-\bar \varphi + t_k v } \mathrm{d} \LL^d
 =   \int_{\Omega}  |\bar \varphi |  \mathrm{d} \LL^d +  t_k  \int_{\Omega}  \bar x v \mathrm{d} \LL^d+    t_k^2   \int_{\ZZ}\ \frac{v^2}{\abs{\nabla \bar \varphi}}\mathrm{d}\mathcal{H}^{d-1} + \oo(t_k^2).
\end{equation}
\end{lemma}
\begin{proof}[Proof in the case $d=1$]
To give the reader an idea of how \Cref{lemma:absoluteTaylorexpansion}
is obtained, we prove \eqref{eq:nonstandardtaylorexpansion} in the one-dimensional setting.
The proof of the general case is similar but much more technical, see \cite[Corollary 5.10]{ChristofMeyerVI}. 
So let us suppose that $d=1$ and that $v \in C_c (\Omega) \cap H^1(\Omega)$ and $\{t_k\} \subset (0, \infty)$ with $t_k \searrow 0$ are given. Then, $\Omega$ is an interval and the compactness of the support 
$\mathrm{supp}(v)$, the regularity of $\bar \varphi$ and our assumption 
$\ZZ \subset \{z \in \Omega : \abs{ \bar \varphi' (z)} \neq 0 \}$
yield that the set $\ZZ \cap \mathrm{supp}(v)$ is finite. 
Denote the elements of $\ZZ \cap \mathrm{supp}(v)$  with $a_i$, $i=1,...,n$, 
assume that $a_1 < a_2 < ... < a_n$ holds and choose $b_i$, $i=1,..., n+1$, such that
$b_i < a_i < b_{i+1}$ for all $i=1,..,n$ and $\mathrm{supp}(v) \subseteq [b_1, b_{n+1}] \subset \Omega$.
Then, we may write
\begin{equation}
\label{eq:integraldecompositiontaylor}
\frac{1}{t_k}\int_\Omega \frac{\abs{-\bar \varphi + t_k v } - \abs{\bar \varphi}}{t_k}
- \bar x v \, \mathrm{d} \LL^1 = 
\sum_{i=1}^n 
\frac{1}{t_k}\int_{b_i}^{b_{i+1}} \frac{\abs{-\bar \varphi + t_k v } - \abs{\bar \varphi}}{t_k}
+\mathrm{sgn}(\bar \varphi)v  \,
\mathrm{d}\LL^1.
\end{equation}
Consider now an arbitrary but fixed $i \in \{1,..., n\}$, assume w.l.o.g.\ that $\bar \varphi' (a_i) > 0$ (the case $\bar \varphi' (a_i) < 0$
is analogous) and choose an
$\varepsilon > 0$ such that $\bar \varphi' \geq \delta > 0$ holds in $(a_i - \varepsilon, a_i + \varepsilon) \subseteq (b_i, b_{i+1})$.
Then, it follows from our construction, the boundedness of $v$, the fact that $\bar \varphi \leq -c < 0$ and $0 < c \leq \bar \varphi$ 
holds in $[b_i, a_i - \varepsilon]$ and $[a_i + \varepsilon, b_{i+1}]$ for some $c>0$, respectively, 
and a simple distinction of cases that
\begin{equation}
\label{eq:randomformula42}
\begin{aligned}
&\frac{1}{t_k}\int_{b_i}^{b_{i+1}} \frac{\abs{-\bar \varphi + t_k v } - \abs{\bar \varphi}}{t_k}
+\mathrm{sgn}(\bar \varphi)v  \,
\mathrm{d}\LL^1 
\\
&=
 \int_{b_i}^{a_i}  2 \frac{\max(0, \bar \varphi - t_k v)  }{t_k^2}
\,
\mathrm{d}\LL^1
+
 \int_{a_i}^{b_{i+1}} 2 \frac{\max(0,-\bar \varphi + t_k v) }{t_k^2}
   \,
\mathrm{d}\LL^1
\\
&=
 \int_{a_i - \varepsilon}^{a_i}  2 \frac{\max(0, \bar \varphi - t_k v)  }{t_k^2}
\,
\mathrm{d}\LL^1
+
 \int_{a_i}^{a_i + \varepsilon} 2 \frac{\max(0,-\bar \varphi + t_k v) }{t_k^2}
   \,
\mathrm{d}\LL^1 + \oo(1),
\end{aligned}
\end{equation}
where the Landau symbol refers to the limit $k \to \infty$. Note that the integrand of the second integral 
on the right-hand side of \eqref{eq:randomformula42} is only non-zero in a $z \in (a_i, a_i + \varepsilon)$ if
\begin{equation*}
\delta (z - a_i) \leq \int_{a_i}^{z} \bar \varphi'(s) \mathrm{d}\LL^1(s) = \bar \varphi(z) \leq  t_k v(z)  \leq t_k \|v\|_{L^\infty(\Omega)}
.
\end{equation*}
Consequently, for all large enough $k$, we have
\begin{equation*}
\begin{aligned}
&\int_{a_i}^{a_i + \varepsilon} 2 \frac{\max(0,-\bar \varphi + t_k v) }{t_k^2}
\mathrm{d}\LL^1
\\
&\quad=
 \int_{a_i}^{a_i + C t_k} 2 \frac{\max(0,-\bar \varphi(z) + t_k v(z)) }{t_k^2}
\mathrm{d}\LL^1(z)
\\
&\quad=
 \int_{0}^{C} 2 \frac{\max(0,-\bar \varphi(a_i + t_k z) + t_k v(a_i + t_k z)) }{t_k^2}
t_k \mathrm{d}\LL^1(z)
\\
&\quad=
 \int_{0}^{C}  2 \max\left (0,- \int_0^1 \bar \varphi'(a_i + s t_k z)  \mathrm{d}s z +  v(a_i + t_k z) \right )
  \mathrm{d}\LL^1(z)
\\
&\quad=
 \int_{0}^{C}  2 \max\left (0,-   \bar \varphi'(a_i  )    z +  v(a_i  ) \right )
  \mathrm{d}\LL^1(z) + \oo(1),
\end{aligned}
\end{equation*}
where $C := \|v\|_{L^\infty(\Omega)}/\delta$ and where the last identity follows from the dominated convergence theorem. 
From $\bar \varphi'(a_i  ) > 0$, we now obtain
\begin{equation*}
\begin{aligned}
&\int_{0}^{C}  2 \max\left (0,-   \bar \varphi'(a_i  )    z +  v(a_i  ) \right )
  \mathrm{d}\LL^1(z) 
\\
&\quad= 
2 \int_{0}^{\max(0, v(a_i))/\bar \varphi'(a_i)  }   -   \bar \varphi'(a_i  )    z +  v(a_i  )  
  \mathrm{d}\LL^1(z)
=
 \frac{\max(0, v(a_i))^2}{\bar \varphi'(a_i)}.
\end{aligned}
\end{equation*}
If we use exactly the same argumentation for the first integral on the right-hand side of \eqref{eq:randomformula42}
and combine our results with \eqref{eq:integraldecompositiontaylor}, then we arrive at the identity
\begin{equation*}
\frac{1}{t_k}\int_\Omega \frac{\abs{-\bar \varphi + t_k v } - \abs{\bar \varphi}}{t_k}
- \bar x v \, \mathrm{d} \LL^1 = 
\sum_{i=1}^n 
\frac{v(a_i)^2}{|\bar \varphi'(a_i)|} + \oo(1).
\end{equation*}
Rewriting the above yields \eqref{eq:nonstandardtaylorexpansion} in the case $d=1$. This completes the proof. 
\end{proof}

The next lemma provides a link between the curvature of $C$
and $\nabla\bar\varphi$.

\begin{lemma}
\label{lemma:fundamentalestimate}
For all $v \in C_c (\Omega)$, all $h \in  \KK_C^{\star}(\bar x, \bar \varphi)$, and all $\alpha > 0$, it holds
\begin{equation}
\label{eq:importantestimate}
\frac{\alpha^2}{2} Q_C^{\bar x, \bar \varphi}( h ) - \alpha \left \langle v,  h \right \rangle_{C_0(\Omega),\MM(\Omega)} +  \int_{\ZZ}\ \frac{v^2}{\abs{\nabla \bar \varphi}}\mathrm{d}\mathcal{H}^{d-1}      \geq  0.
\end{equation}
\end{lemma}

\begin{proof}
Let $\alpha > 0$ and $h \in \KK_C^{\star}(\bar x, \bar \varphi)$ be given. From the cone property of $\KK_C^{\star}(\bar x, \bar \varphi)$, it follows  $\alpha h \in \KK_C^{\star}(\bar x, \bar \varphi)$. This implies  that there exist sequences $\{r_k\} \subset X$, $\{t_k\} \subset \R^+$ with $t_k \searrow 0$, $\smash{t_k \, r_k \weaklystar 0}$ and $\smash{\bar x + t_k \, \alpha h + \frac12 \, t_k^2 \, r_k \in C}$. Fix such sequences  $\{r_k\}$, $\{t_k\}$ and define $h_k := \alpha h + \frac12 \, t_k\, r_k$. Then, it holds $h_k \in \RR_C(\bar x) \subseteq L^\infty(\Omega)$ and $h_k \weaklystar \alpha h$ in $X$, and \Cref{lemma:absoluteTaylorexpansion} yields
\begin{align}
& t_k^2   \int_{\ZZ}\ \frac{v^2}{\abs{\nabla \bar \varphi}}\mathrm{d}\mathcal{H}^{d-1} + \oo(t_k^2)
 \notag
 \\
& \ =   \left ( \sup_{q \in L^\infty(\Omega, [-1,1])} \int_{\Omega}  (-\bar \varphi + t_k v )q \mathrm{d} \LL^d\right ) - \int_{\Omega} \bigh(){ \abs{\bar \varphi }  + t_k \bar x   v } \mathrm{d} \LL^d
\notag
\displaybreak[2]\\
& \ =   \left ( \sup_{p \in L^\infty(\Omega) \colon \bar x + t_k p \in L^\infty(\Omega, [-1,1])} \int_{\Omega}  (-\bar \varphi + t_k v )(\bar x + t_k p) \mathrm{d} \LL^d\right ) - \int_{\Omega}  \bigh(){ \abs{\bar \varphi }  + t_k \bar x   v }  \mathrm{d} \LL^d
\notag
\displaybreak[2]\\
& \ =   \left ( \sup_{p \in L^\infty(\Omega) \colon \bar x +  t_k p \in L^\infty(\Omega, [-1,1])} \int_{\Omega}  (-\bar \varphi + t_k v )( t_k p) \mathrm{d} \LL^d\right )
\notag
\\
\label{eq:randomchainofinequalities43}
& \ \geq     \int_{\Omega}  (-\bar \varphi + t_k v )(  t_k h_k) \mathrm{d}\LL^d\qquad\qquad  \forall v \in C_c (\Omega) \cap H^1(\Omega).
\end{align}
If we divide \eqref{eq:randomchainofinequalities43} by $t_k^2$ and let $k \to \infty$, then we obtain (with $\dual{\bar \varphi}{h_k}_{C_0(\Omega),\MM(\Omega)} = \dual{\bar \varphi}{\frac12 t_k r_k}_{C_0(\Omega),\MM(\Omega)}$)
\begin{equation*}
\begin{aligned}
&\frac12 \liminf_{k \to \infty}   \left \langle \bar  \varphi, r_k   \right \rangle_{C_0(\Omega),\MM(\Omega)}    - \left \langle v, \alpha h \right \rangle_{C_0(\Omega),\MM(\Omega)} +  \int_{\ZZ}\ \frac{v^2}{\abs{\nabla \bar \varphi }}\mathrm{d}\mathcal{H}^{d-1}\geq 0
\end{aligned}
\end{equation*}
for all $v \in C_c (\Omega) \cap H^1(\Omega)$.
Taking the infimum over all sequences  $\{r_k\}$, $\{t_k\}$, using the positive homogeneity of the functional $Q_C^{\bar x, \bar \varphi}(\cdot)$, and employing a density argument, we now arrive at
\begin{equation*}
\frac{\alpha^2}{2}Q_C^{\bar x, \bar \varphi }(h)  - \left \langle v, \alpha h \right \rangle_{C_0(\Omega),\MM(\Omega)} +  \int_{\ZZ}\ \frac{v^2}{\abs{\nabla \bar \varphi }}\mathrm{d}\mathcal{H}^{d-1}\geq 0 \qquad \forall v \in C_c (\Omega).
\end{equation*}
This is the desired estimate.
\end{proof}

We are now in the position to prove a lower bound for the curvature of $C$.

\begin{proposition}
\label{prop:niceTaylorlemma}
For all $h \in  \KK_C^{\star}(\bar x, \bar \varphi)$ with $Q_C^{\bar x, \bar \varphi}( h ) < \infty$ there exists a $g$ such that  
\begin{gather*}
h = g \mathcal{H}^{d-1}|_\ZZ,\qquad g \in L^1\left ( \ZZ,  \mathcal{H}^{d-1} \right) \cap   L^2 \left ( \ZZ, \abs{\nabla \bar \varphi} \mathcal{H}^{d-1} \right), \notag 
\\
\frac{1}{2} \int_{\ZZ} g^2 \abs{\nabla \bar \varphi}\mathrm{d}\mathcal{H}^{d-1}  \leq   Q_C^{\bar x, \bar \varphi}( h)
.
\end{gather*}
\end{proposition}

\begin{proof}
Let $h \in \KK_C^{\star}(\bar x, \bar \varphi) \subseteq \MM(\Omega)$ with $Q_C^{\bar x, \bar \varphi}( h ) < \infty$ be given. Then,
for all $v \in C_c(\Omega)$ with $v = 0$ on $\ZZ$, we obtain from \eqref{eq:importantestimate} that
\begin{equation*}
\frac{\alpha}{2} Q_C^{\bar x, \bar \varphi}( h ) \geq  \left \langle v, h \right \rangle_{C_0(\Omega),\MM(\Omega)}  \quad \forall \alpha > 0,
\end{equation*}
i.e., $ \left \langle v, h \right \rangle_{C_0(\Omega),\MM(\Omega)}  \le 0$.
Using $\pm v$,
we find
$ \left \langle v, h \right \rangle_{C_0(\Omega),\MM(\Omega)}  = 0$.
Hence, the map 
\begin{equation*}
\tilde{h} :  C_c(\mathcal{Z}) \to \mathbb{R},\qquad \tilde{v} \mapsto \left \langle v, h \right \rangle_{C_0(\Omega),\MM(\Omega)},\quad v \in C_c(\Omega), \ v|_\ZZ = \tilde{v}
\end{equation*}
is well-defined as it is independent of the extension  $v$ of $\tilde{v}$ appearing in its definition. Note that, given a $\tilde{v} \in C_c(\ZZ)$, we can always find a $v \in C_c(\Omega)$ with $v|_\ZZ = \tilde v$, cf.\ the submanifold property of $\ZZ$. From \eqref{eq:importantestimate} with 
\begin{equation*}
\alpha = \left ( \int_{\ZZ}\ \frac{\tilde v^2}{\abs{\nabla \bar \varphi }}\mathrm{d}\mathcal{H}^{d-1}   \right )^{1/2} \beta,\quad \beta > 0 \text{ arbitrary but fixed},
\end{equation*}
it now follows
\begin{equation*}
	\left (\frac{\beta }{2} Q_C^{\bar x, \bar \varphi}( h )   + \frac{1}{\beta} \right ) \left ( \int_{\ZZ}\ \frac{\tilde v^2}{\abs{\nabla \bar \varphi }}\mathrm{d}\mathcal{H}^{d-1}   \right )^{1/2}   \geq    \tilde h(\tilde v) \quad \forall \tilde{v} \in  C_c(\mathcal{Z}) .
\end{equation*}
Now, since $C_c(\ZZ)$ is dense in the Lebesgue-space $ L^2 \left ( \ZZ,   \mathcal{H}^{d-1}/\abs{\nabla \bar \varphi}\right)$,
the functional $\tilde h$ can be uniquely extended.
Thus,
there exists an $f \in  L^2 \left ( \ZZ,   \mathcal{H}^{d-1}/\abs{\nabla \bar \varphi}\right)$ with 
\begin{equation*}
	\left (\frac{\beta }{2} Q_C^{\bar x, \bar \varphi }( h)   + \frac{1}{\beta} \right ) \left ( \int_{\ZZ}\ \frac{\tilde v^2}{\abs{\nabla \bar \varphi }}\mathrm{d}\mathcal{H}^{d-1}   \right )^{1/2}   \geq    \tilde h(\tilde v) =  \int_{\ZZ}\ \frac{\tilde v f}{\abs{\nabla \bar \varphi }}\mathrm{d}\mathcal{H}^{d-1}  \quad \forall \tilde{v} \in  C_c(\mathcal{Z}) .
\end{equation*}
Note that $f$ is independent of $\beta$ due to the density of $C_c(\mathcal{Z}) $ in $  L^2 \left ( \ZZ, \mathcal{H}^{d-1}/\abs{\nabla \bar \varphi}\right)$. We thus arrive at 
\begin{equation*}
	\left \langle v, h\right \rangle_{C_0(\Omega),\MM(\Omega)} = \tilde h( v|_\ZZ) =   \int_{\ZZ}\ \frac{ v f}{\abs{\nabla \bar \varphi }}\mathrm{d}\mathcal{H}^{d-1}  \quad \forall v \in C_c(\Omega)
\end{equation*}
with 
\begin{equation*}
\left ( \int_{\ZZ}\ \frac{f^2}{\abs{\nabla \bar \varphi }}\mathrm{d}\mathcal{H}^{d-1} \right )^{1/2} \leq \left (\frac{\beta }{2} Q_C^{\bar x, \bar \varphi}( h)   + \frac{1}{\beta} \right )\qquad \forall \beta > 0.
\end{equation*}
Choosing  $ \beta = ( 2/ Q_C^{\bar x, \bar \varphi}( h))^{1/2}  $ for $ Q_C^{\bar x, \bar \varphi }( h) > 0$ and $\beta$ arbitrarily large for $ Q_C^{\bar x, \bar \varphi }(h ) = 0$, defining $g := f / \abs{\nabla \bar \varphi}$ and using the density of $C_c(\Omega)$ in $C_0(\Omega)$ now yields the claim. Note that $g \in L^1\left ( \ZZ,  \mathcal{H}^{d-1} \right)$ follows trivially from $h = g  \HH^{d-1}|_\ZZ \in \MM(\Omega)$.
\end{proof} 

Next, we address the reverse estimate to that in \Cref{prop:niceTaylorlemma}.

\begin{lemma}
\label{lemma:curvaturecompactsupport}
Let $g \in C_c(\ZZ)$ be given and let $h:= g  \HH^{d-1}|_\ZZ \in \MM(\Omega)$.
Then, $h$ is an element of the critical cone  $\KK_C^{\star}(\bar x, \bar \varphi)$, it holds
\begin{equation}
\label{eq:formulasecondsubderivative42}
  \frac{1}{2}  \int_{\ZZ}\  g^2 \abs{\nabla \bar \varphi } \mathrm{d}\mathcal{H}^{d-1}
=
Q_C^{\bar x, \bar \varphi }( h ),
\end{equation}
and for every sequence $\{t_k\} \subset \R^+$ with $t_k \searrow 0 $ there exists a sequence $\{r_k\} \subset X$ 
such that $\bar x + t_k h + \frac12 t_k^2 r_k \in C$ holds for all $k$ and  
such that $t_k r_k \weaklystar 0$, $\|h + \frac12 t_k r_k\|_X \to \|h\|_X $ and $\dual{\bar \varphi}{r_k}_{C_0(\Omega),\MM(\Omega)} \to Q_C^{\bar x, \bar \varphi }( h )$  holds for  $k \to \infty$.
\end{lemma}

\begin{proof}
 Since $\bar \varphi \in C_0(\Omega)$ and $\ZZ = \{\bar \varphi = 0\}$, it trivially holds $\dual{\bar \varphi}{h}_{C_0(\Omega),\MM(\Omega)} = 0$ and, consequently,  $h \in \bar \varphi^\perp$. It remains to show that $h$ is 
an element of the \mbox{weak-$\star$} tangent cone $\TT_C^{\star}(\bar x)$, that 
\eqref{eq:formulasecondsubderivative42} holds, and that for each 
$\{t_k\}  \subset \mathbb{R}^+$ with $t_k \searrow 0$ we can find a sequence $\{r_k\}$
with the desired approximation properties. 
To prove these three assertions, we proceed in two steps:

Step 1 (Proof in a Rectification Neighborhood):
In what follows, we first consider a prototypical situation, where the support of the function $g$ is contained in 
a rectification neighborhood of the $C^1$-manifold $\ZZ$, i.e., in an open set where $\ZZ$ resembles a 
$C^1$-graph. As we will see, in this simplified setting, we can manually construct a sequence $\{h_k\}$
such that  $r_k := 2(h_k-h)/t_k$ satisfies the conditions in the lemma, see \eqref{eq:hkprotodef}, \eqref{eq:limit11},
\eqref{eq:limit12} and \eqref{eq:limit22} below.

Let us denote with $\partial_i$, $i=1,...,d$, the partial derivatives of a function and
assume that a point $p \in  \ZZ$, an open ball $B\subset \mathbb{R}^{d-1}$, an open interval $J := (a, b)$, and a map $\gamma \in C^1(\overline{B})$ are given such that 
\begin{equation*}
\begin{aligned}
p &\in  B\times J
,\qquad
\overline{B \times J} \subseteq \Omega
, \qquad
\ZZ \cap   (B \times J)  =  \{(z,\gamma(z)) : z \in B \}
\\
W  &:= \{(z,z') : z \in B , |z' - \gamma (z)| < \varepsilon \} \subseteq B  \times J
, \qquad
\partial_d \bar\varphi(p) > 0
\end{aligned}
\end{equation*}
for some $\varepsilon \in (0,1)$.
Let $0 \leq \psi \in C(W)$ and $g \in C(\ZZ \cap  (B \times J) )$ be continuous and bounded functions and let $t_k \in (0, \infty)$ be a sequence with $t_k \searrow 0$. Assume w.l.o.g.\ that 
\begin{equation*}
t_k \|g\|_{L^\infty}   \| \sqrt{1 + \abs{\nabla \gamma }^2}  \|_{L^\infty} \leq \varepsilon
\end{equation*}
for all $k$ (else consider $\{t_k\}_{k \geq K}$, $K \in \mathbb{N}$ sufficiently large) and extend $\psi$ by zero outside of $W$.
With some abuse of notation, we extend the sign of $g$ from $\ZZ \cap W$ to $W$ by
\begin{equation*}
	\sign(g)(z, z') := \sign g(z, \gamma(z))
	.
\end{equation*}
We further define the sets
\begin{equation*}
	G_k
	:=
	\biggh\{\}{
		(z, z') \in B \times J
		\mid
		t_k \frac{g^-(z, \gamma  (z))}{2 }
		\le
		\frac{z' - \gamma(z)}{\sqrt{1 + \abs{\nabla \gamma (z)}^2}}
		\le
		t_k \frac{g^+(z, \gamma  (z))}{2 }
	}
\subset W
	,
\end{equation*}
where $g^+$ and $g^-$ are abbreviations for $\max(0, g)$ and $\min(0, g)$, respectively,
and set
\begin{equation}
\label{eq:hkprotodef}
 h_k 
 :=
 \frac{2 \, \sign(g)}{t_k} \, \I_{G_k},
\end{equation}
where $\mathds{1}_A : \Omega \to \{0,1\}$ denotes the characteristic function of a set $A \subset \Omega$.
We claim that the above $h_k$ satisfies
\begin{align}
\label{eq:limit11}
\lim_{k \to \infty}\left ( \int_\Omega h_k \psi  v  \mathrm{d}\LL^d  \right ) &=  \int_{\ZZ \cap  (B \times J) }\  g \psi  v \, \mathrm{d}\mathcal{H}^{d-1}  \quad \forall v \in C_0(\Omega),
\\
\label{eq:limit12}
\lim_{k \to \infty}\left ( \int_\Omega |h_k| \psi  \mathrm{d}\LL^d  \right ) &=  \int_{\ZZ \cap  (B \times J) }\  |g| \psi  \mathrm{d}\mathcal{H}^{d-1},
\shortintertext{and}
\label{eq:limit22}
\lim_{k \to \infty}\left ( \int_\Omega  \frac{2h_k}{t_k} \bar \varphi \psi   \mathrm{d}\LL^d  \right ) &=   \frac{1}{2}  \int_{\ZZ \cap  (B \times J)}\  g^2 \psi \abs{\nabla \bar \varphi } \mathrm{d}\mathcal{H}^{d-1}.
\end{align}
This can be seen as follows:  Given a $ v \in C_0(\Omega)$, we may calculate (using the abbreviations $\mathrm{d}z :=\mathrm{d} \LL^{d-1}(z)$ and $\mathrm{d}z' := \mathrm{d}\LL^1(z')$)
\begin{equation*}
\begin{aligned}
	\int_\Omega h_k \psi  v \mathrm{d}\LL^d  
	&=  \frac{2}{t_k} \int_B  \int_{\gamma(z) +  t_k \frac{g^-(z, \gamma(z))}{2 }  \sqrt{1 + \abs{\nabla \gamma (z)}^2}}^{\gamma(z) +  t_k \frac{g^+(z, \gamma(z))}{2 }  \sqrt{1 + \abs{\nabla \gamma (z)}^2}} \sign(g) v \psi \mathrm{d}z' \mathrm{d}z \\
	&=  2 \int_B  \int_{\frac{g^-(z, \gamma (z))}{2 }  \sqrt{1 + \abs{\nabla \gamma (z)}^2}}^{\frac{g^+(z, \gamma (z))}{2 }  \sqrt{1 + \abs{\nabla \gamma (z)}^2}}  ( \sign(g) v \psi  )|_{(z, \gamma(z) + t_k z' )}\mathrm{d}z'\mathrm{d}z \\
	&\to  2 \int_B  \int_{\frac{g^-(z, \gamma  (z))}{2 }  \sqrt{1 + \abs{\nabla \gamma (z)}^2}}^{\frac{g^+(z, \gamma  (z))}{2 }  \sqrt{1 + \abs{\nabla \gamma (z)}^2}}  ( \sign(g) v \psi  )|_{(z, \gamma(z) )}\mathrm{d}z'\mathrm{d}z \\
	&=  \int_B \sqrt{1 + \abs{\nabla \gamma (z)}^2}( g  v \psi  )|_{(z, \gamma(z) )} \mathrm{d}z
	=   \int_{\ZZ \cap  (B \times J) }\  g \psi  v \, \mathrm{d}\mathcal{H}^{d-1} .
\end{aligned}
\end{equation*}
This yields \eqref{eq:limit11}. To obtain \eqref{eq:limit12}, we can use exactly the same calculation as above
(just replace $v$ with $\sign(g)$).
It remains to prove \eqref{eq:limit22}. To this end, we compute
\begin{align*}
& \int_\Omega \frac{2h_k}{t_k} \bar \varphi \psi \mathrm{d}\LL^d
=  \frac{4}{t_k}  \int_B  \int_{\frac{g^-(z, \gamma  (z))}{2 }  \sqrt{1 + \abs{\nabla \gamma (z)}^2}}^{\frac{g^+(z, \gamma  (z))}{2 }  \sqrt{1 + \abs{\nabla \gamma (z)}^2}}  ( \sign(g) \bar \varphi \psi  )|_{(z, \gamma(z) + t_k z' )}\mathrm{d}z'\mathrm{d}z
\displaybreak[2]
\\
&=  4  \int_B  \int_{\frac{g^-(z, \gamma(z))}{2 } \scriptscriptstyle \sqrt{1 + \abs{\nabla \gamma (z)}^2}}^{\frac{g^+(z, \gamma(z))}{2 } \scriptscriptstyle \sqrt{1 + \abs{\nabla \gamma (z)}^2}}  \int_0^{1} \partial_d \bar \varphi (z, \gamma(z) + s t_k z' )  \mathrm{d}s  \, (\sign(g)\psi)|_{(z, \gamma(z) + t_k z' )} z'\mathrm{d}z'\mathrm{d}z
\displaybreak[2]
\\
&\to  4  \int_B  \int_{\frac{g^-(z, \gamma(z))}{2 }  \sqrt{1 + \abs{\nabla \gamma (z)}^2}}^{\frac{g^+(z, \gamma(z))}{2 }  \sqrt{1 + \abs{\nabla \gamma (z)}^2}}  z'\mathrm{d}z'(\sign(g)\partial_d \bar \varphi \, \psi)|_{(z, \gamma(z) )} \, \mathrm{d}z
\\
&= \frac{1}{2} \int_B  \left (    1 + \abs{\nabla \gamma (z)}^2  \right )    \, (  g^2 \partial_d \bar \varphi \,  \psi)|_{(z, \gamma(z ))}\mathrm{d}z.
\end{align*}
Differentiating $\bar\varphi(z,\gamma(z))$
w.r.t.\ $z_i$ yields
$\partial_d \bar \varphi(z,\gamma(z)) \, \partial_i \gamma(z) = - \partial_i \bar \varphi(z,\gamma(z))$ for all $i = 1,\ldots, d-1$.
Thus,
\begin{equation*}
	\partial_d \bar \varphi(z, \gamma(z))\sqrt{1 + \abs{\nabla \gamma (z)}^2} = \abs{\nabla \bar \varphi(z, \gamma(z))}\quad \forall z \in B.
\end{equation*}
Hence, we arrive at 
\begin{equation*}
\begin{aligned}
\int_\Omega \frac{2h_k}{t_k} \bar \varphi \psi \mathrm{d}\LL^d
&\to  \frac{1}{2} \int_B  \sqrt{  1 + \abs{\nabla \gamma (z)}^2}  \, ( \abs{\nabla \bar \varphi } g^2 \psi  )|_{(z, \gamma(z) )}\mathrm{d}z
\\
&= \frac{1}{2}  \int_{\ZZ \cap  (B \times J)}\  g^2 \psi \abs{\nabla \bar \varphi } \mathrm{d}\mathcal{H}^{d-1}.
\end{aligned}
\end{equation*}
This proves that \eqref{eq:limit22} holds and that $\{h_k\}$ indeed has the desired properties.

Step 2 (Proof in the General Case):
In this second part of the proof, we demonstrate that, given an arbitrary 
but fixed $h:= g  \HH^{d-1}|_\ZZ \in \MM(\Omega)$, $g \in C_c(\ZZ)$, we can always use 
a partition of unity and the manifold property of $\ZZ$ to reduce the situation 
to the case studied in Step~1, see \eqref{eq:randomconvergence251524i}, \eqref{eq:randomconvergence251524ii} and \eqref{eq:randomconvergence251524iii} below.

Recall that the implicit function theorem and the definition of $\ZZ$ imply that for every $p \in  \ZZ$  there exist an orthogonal transformation $R  \in O(d)$, an open ball $B \subset \mathbb{R}^{d-1}$, an open interval $J := (a, b)$, and a map $\gamma \in C^1(\overline{B})$ with values in $J $ such that 
\begin{gather*}
p \in R( B\times J),\qquad \overline{R( B\times J)}   \subseteq \Omega,\qquad
\ZZ \cap  R(B \times J)  = R( \{(z,\gamma(z)) : z \in B \}),\\
W  := \{(z,z') : z \in B , |z' - \gamma (z)| < \varepsilon \} \subseteq B  \times J ,\\
R(\{(z,z') \in B \times J: \gamma (z) < z' \})  \subseteq \{\bar \varphi > 0\},\\
R(\{(z,z') \in B \times J: z' < \gamma (z) \})  \subseteq \{\bar \varphi < 0\}
\end{gather*}
for some $\varepsilon > 0$. Since $\mathrm{supp}(g)$ is compact, we may find points $p_1,\ldots,p_L \in \ZZ$ with associated $R_l, B_l$ etc.\ such that the sets $R_l(W_l)$, $l=1,\ldots,L$, cover $\mathrm{supp}(g)$. Define
\begin{equation*}
U := \bigcup_{l=1}^L R_l(W_l) \subseteq \Omega
\end{equation*}
and choose a partition of unity $(\psi_l)_{l=1}^L$ subordinate to the $R_l(W_l)$-cover of the set $U$, i.e., a collection of continuous functions $\psi_l \in C(U, [0, 1])$ such that 
\begin{equation*}
\mathrm{supp}(\psi_l) \subseteq W_{l} ,\quad \sum_l \psi_l = 1 \text{ on } U.
\end{equation*}
Consider now an arbitrary but fixed sequence $t_k \in (0, \infty)$ with $t_k \searrow 0$, extend the functions $\psi_l$ by zero and define (for $k$ large enough)
\begin{equation*}
\begin{aligned}
	&\sign_l(g)(R_l(z,z')) := \sign g\bigh(){R_l(z,\gamma_l(z))} \qquad \forall (z,z') \in W_l,
	\\
	&G_{l,k}
	:=
	\biggh\{\}{
		(z, z') \in B_l \times J_l
		\mid
		\frac{g^-(R_l(z, \gamma_l  (z)))}{2 }
		\le
		\frac{t_k^{-1} (z' - \gamma_l(z))}{\sqrt{1 + \abs{\nabla \gamma_l (z)}^2}}
		\le
		\frac{g^+(R_l(z, \gamma_l  (z)))}{2 }
	}
	,
	\\
	&h_k := \sum_{l = 1}^L \frac{2 \, \sign_l(g)}{t_k} \, \I_{R_l (G_{l,k})} \psi_l
	.
\end{aligned}
\end{equation*}
Then, it holds $\bar x + t_k h_k \in   C = L^\infty(\Omega, [0,1])$ (cf.\ the signs of the involved functions and $\|h_k\|_{L^\infty} \leq 2/t_k$), and we may deduce from (b) that
for all $v \in C_0(\Omega)$, we have
\begin{align}
\label{eq:randomconvergence251524i}
\lim_{k \to \infty}\left ( \int_\Omega h_k v  \mathrm{d}\LL^d  \right ) &=  \sum_{l=1}^L \int_{ \ZZ \cap  R_l(B_l \times J_l) }\  g \psi_l  v \, \mathrm{d}\mathcal{H}^{d-1} = \left \langle v, h \right \rangle_{C_0(\Omega),\MM(\Omega)}
\displaybreak[2]
\\
\label{eq:randomconvergence251524ii}
\lim_{k \to \infty}\left ( \int_\Omega |h_k|  \mathrm{d}\LL^d  \right ) &=  \sum_{l=1}^L  \int_{\ZZ \cap  R_l(B_l \times J_l) }\  |g| \psi_l  \mathrm{d}\mathcal{H}^{d-1} = \|h\|_X
\displaybreak[2]
\shortintertext{and}
\label{eq:randomconvergence251524iii}
\lim_{k \to \infty}\left ( \int_\Omega  \frac{2h_k}{t_k} \bar \varphi  \mathrm{d}\LL^d  \right )
&
=   \frac{1}{2}  \sum_{l=1}^L  \int_{ \ZZ \cap  R_l(B_l \times J_l) } g^2 \psi_l \abs{\nabla \bar \varphi } \mathrm{d}\mathcal{H}^{d-1}
\\&\notag
= \frac{1}{2}  \int_{\ZZ}\  g^2 \abs{\nabla \bar \varphi } \mathrm{d}\mathcal{H}^{d-1}.
\end{align}
The above proves $h \in  \TT_C^{\star}(\bar x)$ and $h \in \KK_C^{\star}(\bar x, \bar \varphi)$, see (a). From \eqref{eq:randomconvergence251524i}, \eqref{eq:randomconvergence251524ii} and \eqref{eq:randomconvergence251524iii}, we obtain further that $r_k := 2(h_k-h)/t_k$ satisfies $\bar x + t_k h + \frac12 t_k^2 r_k \in C$ for all $k$ and
\begin{equation*}
	t_k r_k \weaklystar 0, \quad   \|h + \frac12 t_k r_k\|_X \to \|h\|_X, \quad \dual{\bar \varphi}{r_k}_{C_0(\Omega),\MM(\Omega)} \to \frac{1}{2}  \int_{\ZZ}\  g^2 \abs{\nabla \bar \varphi } \mathrm{d}\mathcal{H}^{d-1}
\end{equation*}
as $k \to \infty$. If we combine this with \Cref{prop:niceTaylorlemma} and \Cref{def:curvature_term}, then the claim follows immediately. 
\end{proof}
Using \Cref{prop:niceTaylorlemma} and \Cref{lemma:curvaturecompactsupport}, we finally arrive at an explicit formula for the directional curvature functional in the bang-bang case.

\begin{theorem}
\label{th:explicitcurvature}
For every tuple $(\bar x, \bar \varphi) \in C \times - \NN_C^\star(\bar x)$ that satisfies the conditions in \Cref{assumption:bangbangsolution}, it holds
\begin{equation}
	\begin{aligned}
&\varh\{\}{
h \in  \KK_C^{\star}(\bar x, \bar \varphi)  \mid Q_C^{\bar x, \bar \varphi}( h )  < \infty
}
\\&\qquad
=
\varh\{\}{
	g \mathcal{H}^{d-1}|_\ZZ \mid g \in L^1\left ( \ZZ,  \mathcal{H}^{d-1} \right) \cap   L^2 \left ( \ZZ, \abs{\nabla \bar \varphi} \mathcal{H}^{d-1} \right)
}.
	\end{aligned}
\label{eq:inclusion42}
\end{equation} 
Moreover, for every element $h = g \mathcal{H}^{d-1}|_\ZZ$ of the above set, it is true that
\begin{equation}
Q_C^{\bar x, \bar \varphi }( h ) =
\frac{1}{2}  \int_{\ZZ}\  g^2 \abs{\nabla \bar \varphi } \mathrm{d}\mathcal{H}^{d-1}.
\label{eq:curvatureidentity42}
\end{equation}
\end{theorem}

\begin{proof}
	\Cref{prop:niceTaylorlemma} yields that ``$\subseteq$'' holds in \eqref{eq:inclusion42}
	and that ``$\geq$'' holds in \eqref{eq:curvatureidentity42}.
	To obtain the reverse inclusion/inequality,
	we consider an arbitrary but fixed $h =  g \mathcal{H}^{d-1}|_\ZZ \in \MM(\Omega)$
	with some $g \in L^1\left ( \ZZ,  \mathcal{H}^{d-1} \right) \cap   L^2 \left ( \ZZ, \abs{\nabla \bar \varphi} \mathcal{H}^{d-1} \right)$.
	Using a compact exhaustion of the domain $\Omega$ and mollification,
	it is easy to see that we can find a sequence $(g_n) \subset C_c(\ZZ)$
	with $g_n \to g$ in $ L^1\left ( \ZZ,  \mathcal{H}^{d-1} \right ) \cap L^2 \left ( \ZZ, \abs{\nabla \bar \varphi} \mathcal{H}^{d-1} \right)$.
	The latter implies that $h_n := g_n \mathcal{H}^{d-1}|_\ZZ $ satisfies $\smash{h_n \weaklystar h}$ in $X$ and that
\begin{equation*}
	\frac{1}{2}  \int_{\ZZ}\  g_n^2 \abs{\nabla \bar \varphi } \mathrm{d}\mathcal{H}^{d-1} \to   \frac{1}{2}  \int_{\ZZ}\  g^2 \abs{\nabla \bar \varphi } \mathrm{d}\mathcal{H}^{d-1}.
\end{equation*}
On the other hand, we obtain from \Cref{lemma:curvaturecompactsupport}
that $h_n \in \KK_C^{\star}(\bar x, \bar \varphi)$ holds for all $n$
and that there exist $\{r_{n,k}\} \subset X$
and $\{t_{n,k}\} \subset \mathbb{R}^+$ with $\smash{\bar x + t_{n,k} \, h_n + \frac12 \, t_{n,k}^2 \, r_{n,k} \in C}$ for all $n, k$ and 
\begin{align*}
&
t_{n,k} \searrow 0,\quad   t_{n,k} r_{n,k} \weaklystar 0,\quad \dual{\bar \varphi}{r_{n,k}}_{C_0(\Omega),\MM(\Omega)} \to Q_C^{\bar x, \bar \varphi }( h_n )
\\
&\qquad\qquad
\text{and}\quad \|h_n + \frac12 t_{n,k} r_{n,k}\|_X \to \|h_n\|_X 
\end{align*}
for all $n$ as $k \to \infty$.
Note that, since the sequence $\|h_n\|_X$ is bounded
and since the norms $\|h_n + \frac12 t_{n,k} r_{n,k}\|_X $ converge to $\|h_n\|_X$,
in the above situation, we may assume w.l.o.g.\ that $\|t_{n,k}r_{n,k}\|_X \leq M$ holds
for some constant $M$ that is independent of $n$ and $k$ (just shift the index $k$ appropriately for each $n$).
From \Cref{lem:curvature_uhs,lemma:curvaturecompactsupport},
we may now deduce that $h$ is an element of the critical cone $\KK_C^{\star}(\bar x, \bar \varphi)$ and that
\begin{equation*}
	Q_C^{\bar x, \bar \varphi}(h) \leq \liminf_{n \to \infty} Q_C^{\bar x, \bar \varphi}(h_n) =  \liminf_{n \to \infty} \frac{1}{2}  \int_{\ZZ}\  g_n^2 \abs{\nabla \bar \varphi } \mathrm{d}\mathcal{H}^{d-1} =  \frac{1}{2}  \int_{\ZZ}\  g^2 \abs{\nabla \bar \varphi } \mathrm{d}\mathcal{H}^{d-1} < \infty.
\end{equation*}
This proves ``$\supseteq$'' in \eqref{eq:inclusion42} and  ``$\leq$'' in \eqref{eq:curvatureidentity42}. 
\end{proof}

We combine \Cref{thm:no-gap-measures,th:explicitcurvature}
and arrive at the main result of this section.

\begin{theorem}[Explicit No-Gap Second-Order Condition for Bang-Bang Problems]
	\label{thm:explicitno-gap-measures}
	Consider an optimization problem of the form \eqref{eq:problem}
	with $C$, $X$ etc.\  as in \Cref{assumption:bangbangsetting}. Assume that $\bar x \in C$ is fixed, that $J$ satisfies the conditions in  \Cref{asm:standing_SOC}, that the map
	$ X \ni h \mapsto J''(\bar x) \, h^2 \in \R$ is  \mbox{weak-$\star$} continuous, that $\bar x$ and $\bar \varphi$ satisfy the conditions  in \Cref{assumption:bangbangsolution}, and that the constant $K(\bar \varphi)$ in \eqref{eq:measure_of_set} is non-zero. 
	Then, the condition
	\begin{equation}
		\label{eq:explicitbangbang}
		\begin{aligned}
			& \frac{1}{2}  \int_{\ZZ}\  g^2 \abs{\nabla \bar \varphi } \mathrm{d}\mathcal{H}^{d-1}  + J''(\bar x) \left ( g \mathcal{H}^{d-1}|_\ZZ  \right )^2
		>
		0
		\\
		&\qquad\qquad\qquad\qquad \forall  g \in  L^1\left ( \ZZ,  \mathcal{H}^{d-1} \right) \cap   L^2 \left ( \ZZ, \abs{\nabla \bar \varphi} \mathcal{H}^{d-1} \right) \setminus \{0\}
		\end{aligned}
	\end{equation}
	is equivalent to the quadratic growth condition \eqref{eq:bangbanggrowth} with constants $c>0$ and $\varepsilon > 0$.
\end{theorem}

We conclude this section with some remarks.
\begin{remark}~
\begin{enumerate}
\item The inequality \eqref{eq:explicitbangbang} with ``$\geq$'' instead of ``$>$'' is still a necessary optimality condition when $K(\bar \varphi ) = 0$.
\item \Cref{thm:explicitno-gap-measures} is still valid when $\mathcal{Z} = \emptyset$ and \eqref{eq:explicitbangbang} is empty. In this case, the only thing that has to be checked to obtain  \eqref{eq:bangbanggrowth}  is the level-set condition $K(\bar \varphi) > 0$. 
\item We expect that the assumptions on $\bar \varphi$ in \Cref{thm:explicitno-gap-measures} can be weakened (the Taylor expansion in \Cref{lemma:absoluteTaylorexpansion}, for example, also holds in a far more general setting, see \cite[Proposition 5.9]{ChristofMeyerVI}).  
\item The techniques used in the proofs of \Cref{lemma:fundamentalestimate,prop:niceTaylorlemma} might be useable in other situations as well (e.g., the idea to exploit the subdifferential structure of the admissible set $C$, cf.\ \eqref{eq:randomchainofinequalities43}).
\item We point out that \eqref{eq:MRC} does not hold in the situation of  \Cref{thm:explicitno-gap-measures} (from $C \subseteq L^1(\Omega)$ it follows $\TT_C(\bar x) \cap \bar \varphi\anni = \{0\} \ne \KK_C^\star(\bar x, \bar \varphi)$ and this is incompatible with the condition \eqref{eq:MRC}, cf.\ \Cref{remark_mrc}).
\item The first addend on the left-hand side of \eqref{eq:explicitbangbang} measures the curvature of the set $C$ in $\bar x$ (compare, e.g.,  with \Cref{thm:no-gap-SOC_soreg}).
	Note that the surface integral in \eqref{eq:explicitbangbang} is only meaningful for a $C^1$-function $\bar \varphi$. This shows that exploiting the regularity of the gradient $\bar \varphi = J'(\bar x)$ is essential for the derivation of \Cref{thm:explicitno-gap-measures}.
\end{enumerate}
\end{remark}

Finally, we would like to compare our \Cref{thm:explicitno-gap-measures}
with the results of
\cite{CasasWachsmuthsBangBang}
by means of an example.

\begin{example}
	\label{ex:semilinear_bang_bang}
	We consider the optimal control problem
	\begin{equation}
		\label{eq:optimal_control_example}
		\begin{aligned}
			\text{Minimize} \quad & \frac12 \, \norm{ y - y_d }_{L^2(\Omega)}^2 \\
			\text{such that} \quad & -\Delta y + a(y) = u \quad\text{ in } \Omega,
			\quad
			y = 0 \quad\text{ on } \partial\Omega\\
			\text{and}\quad & -1 \le u \le 1 \quad\text{ in } \Omega.
		\end{aligned}
	\end{equation}
	Here, the state equation is to be understood in the weak sense.
	We assume that $\Omega \subset \R^d$, $d \in \{1,2,3\}$, is a bounded domain with Lipschitz boundary $\partial\Omega$,
	that
	$y_d \in L^2(\Omega)$ and that $a : \R \to \R$ is twice continuously differentiable and monotonically increasing.
	Further, we define the admissible set $C := L^\infty(\Omega; [-1,1])$.

	Now, we can employ \cite[Theorem~2.1]{CasasWachsmuthsBangBang} to
	obtain that the control-to-state operator $G : L^2(\Omega) \to H_0^1(\Omega) \cap C(\bar\Omega)$
	which maps $u$ to the solution $y$ of the semilinear PDE in \eqref{eq:optimal_control_example}
	is well defined
	and twice continuously Fréchet differentiable.
	For $\bar u, h \in L^2(\Omega)$, the derivative $z_h := G'(\bar u) h$ solves the linearized PDE
	\begin{equation*}
		-\Delta z_h + a'(\bar y) z_h = h \quad\text{ in } \Omega,
		\quad
		z_h = 0 \quad\text{ on } \partial\Omega,
	\end{equation*}
	with $\bar y = G\bar u$.
	Consequently, we can check that the reduced objective $J : L^\infty(\Omega) \to \R$,
	$J(u) = \frac12 \, \norm{G u - y_d}_{L^2(\Omega)}^2$ is twice continuously Fréchet differentiable
	and that the derivatives are given by
	\begin{equation*}
		J'(\bar u) \, h = \int_\Omega \bar\varphi h \,  \mathrm{d} \LL^d,
		\qquad
		J''(\bar u) \, h^2 = \int_\Omega\Bigh[]{1 - a''(\bar y) \,\bar\varphi }\, z_h^2 \,  \mathrm{d} \LL^d.
	\end{equation*}
	Here, $\bar\varphi \in H_0^1(\Omega) \cap C(\bar\Omega)$ is the solution of the adjoint equation
	\begin{equation*}
		-\Delta \bar\varphi + a'(\bar y) \bar\varphi = \bar y - y_d \quad\text{ in } \Omega,
		\quad
		\bar\varphi = 0 \quad\text{ on } \partial\Omega.
	\end{equation*}
	In order to apply the previous theorems, we have to check that
	the Taylor expansion \eqref{eq:hadamard_taylor_expansion} from \Cref{asm:standing_SOC}
	is satisfied by $J$ with the setting of spaces as in \Cref{assumption:bangbangsetting}.
	First of all, we mention that the quadratic form $J''(\bar u)$ can be extended
	from $L^\infty(\Omega)^2$ to $\MM(\Omega)^2$
	by using results for PDEs with measures on the right-hand side,
	see the discussion in \cite[Section~2.5]{CasasWachsmuthsBangBang}.
	Now, let sequences $\{h_k\} \subset L^\infty(\Omega)$, $\{t_k\} \subset \R^+$
	with $t_k \searrow 0$, $h_k \weaklystar h$ in $\MM(\Omega)$
	and $-1 \le \bar u + t_k \, h_k \le 1$ be given.
	Then, the twice continuous Fréchet differentiability
	of $J : L^\infty(\Omega) \to \R$
	yields
	\begin{equation*}
		J(\bar u + t_k \, h_k)
		=
		J(\bar u)
		+
		t_k \, J'(\bar u) \, h_k
		+
		\frac12 \, t_k^2 \, J''(\bar u + \theta_k \, t_k \, h_k) \, h_k^2
	\end{equation*}
	for some $\theta_k \in [0,1]$.
	Further, \cite[Lemmas~2.6, 2.7]{Casas2012:1} imply that for all $\varepsilon > 0$,
	the inequality
	\begin{equation*}
		\bigabs{
			J''(\bar u) \, h_k^2 - J''(\bar u + \theta_k \, t_k \, h_k) \, h_k^2
		}
		\le
		\varepsilon \, \norm{z_{h_k}}_{L^2(\Omega)}^2
		\le
		C \, \varepsilon \, \norm{h_k}_{L^1(\Omega)}^2
		\le
		\hat C \, \varepsilon
	\end{equation*}
	holds if $k$ is large enough.
	Together with the above Taylor expansion, we find
	\begin{align*}
		\frac{
			J(\bar u + t_k \, h_k)
			-
			J(\bar u)
			-
			t_k \, J'(\bar u) \, h_k
			-
			\frac12 \, t_k^2 \, J''(\bar u) \, h_k^2
		}{t_k^2}
		&=
		\frac{
			J''(\bar u + \theta_k \, t_k \, h_k) \, h_k^2
			-
			J''(\bar u) \, h_k^2
		}{2}
		\\
		&\to 0
	\end{align*}
	as $k \to \infty$.
	This shows that \Cref{asm:standing_SOC} is satisfied.

	In order to proceed,
	we fix $\bar u \in C$ such that the first-order condition
	$J'(\bar u) \, (u - \bar u) \ge 0$ for all $u \in C$
	is satisfied.
	We further assume that the adjoint state $\bar\varphi$
	has the additional regularity $\bar\varphi \in C^1(\bar\Omega)$
	and we suppose that
	\begin{equation}
		\label{eq:non-vanishing_gradient}
		\min_{x \in N} \abs{\nabla\bar\varphi(x)} > 0,
		\quad
		\text{where } N := \{x \in \bar\Omega : \bar\varphi(x) = 0 \}.
	\end{equation}
	On the one hand, this condition ensures
	that \Cref{assumption:bangbangsolution} holds.
	On the other hand,
	from
	\cite[Lemma~3.2]{DeckelnickHinze2012}
	we know that this condition implies
	\begin{equation*}
		\LL^d(\{ |\bar\varphi| \leq s\})
		\le
		\hat K \, s
		\qquad\forall s > 0
	\end{equation*}
	for some $\hat K > 0$.
	From the definition \eqref{eq:measure_of_set}, we directly infer $K(\bar\varphi) \ge (4 \,\hat K)^{-1} > 0$,
	such that we can apply \Cref{thm:no-gap-measures,thm:explicitno-gap-measures}.

	We start by the interpretation of \Cref{thm:no-gap-measures}
	and compare it to the results of \cite{CasasWachsmuthsBangBang,CasasWachsmuthsBangBang2}.
	As discussed in \Cref{rem:no-gap_measure},
	we obtain that the condition
	\begin{equation}
		\label{eq:suff_cond_bang_bang}
		J''(\bar u) \, h^2 > - K(\bar \varphi)\, \norm{h}^2_{\MM(\Omega)}
		\quad
		\forall h \in \KKs_C(\bar u, \bar \varphi ) \setminus \{0\}
	\end{equation}
	implies a quadratic growth at $\bar u$ in $L^1(\Omega)$,
	and this sufficient condition is similar to \cite[Corollary~2.15]{CasasWachsmuthsBangBang}.
	Here, $\KKs_C(\bar u, \bar\varphi)$ is the critical cone of $C = L^\infty(\Omega; [-1,1])$
	w.r.t.\ the weak-$\star$ topology in $\MM(\Omega)$, see \Cref{eq:def_cones}.
	In our situation, we have $\KKs_C(\bar u, \bar\varphi) = \MM(\ZZ)$
	with $\ZZ = \{ x \in \Omega : \bar\varphi(x) = 0\}$,
	see \cite[Corollary~2.13]{CasasWachsmuthsBangBang} and \eqref{eq:non-vanishing_gradient}.
	Note that \cite[Theorem~2.4]{CasasWachsmuthsBangBang2}
	improves the constant in \cite[Theorem~2.8]{CasasWachsmuthsBangBang},
	and therefore,
	also \cite[Corollary~2.15]{CasasWachsmuthsBangBang}
	can be improved slightly,
	see also \Cref{rem:no-gap_measure}~\ref{rem:no-gap_measure_iii}.
	However, it is not possible to obtain a necessary optimality condition similar to \eqref{eq:suff_cond_bang_bang}.

	By applying \Cref{thm:explicitno-gap-measures}, we obtain that
	\begin{equation}
		\label{eq:no-gap_cond_bang_bang}
		\frac{1}{2}  \int_{\ZZ}\  g^2 \abs{\nabla \bar \varphi } \mathrm{d}\mathcal{H}^{d-1}  + J''(\bar u) \left ( g \mathcal{H}^{d-1}|_\ZZ  \right )^2
		>
		0
		\qquad
		\forall  g \in  L^2 \left ( \ZZ, \mathcal{H}^{d-1} \right) \setminus \{0\}
	\end{equation}
	is equivalent to a quadratic growth of $J$ at $\bar u$ w.r.t.\ the norm in $L^1(\Omega)$.
	Note that due to \eqref{eq:non-vanishing_gradient},
	we do not need to employ weighted Lebesgue spaces on $\ZZ$.
	
	It is interesting to see that the sufficient condition \eqref{eq:suff_cond_bang_bang}
	involves all signed finite Radon measures  with support contained in $\ZZ$,
	whereas we only
	need to consider Radon measures with $L^2$-density (w.r.t.\ $\HH^{d-1}$)
	on $\ZZ$ for the no-gap condition \eqref{eq:no-gap_cond_bang_bang}.

	To conclude,
	our technique of performing a careful analysis of the curvature of
	the feasible set $C = L^\infty(\Omega; [-1,1]) \subset \MM(\Omega)$
	allows us to formulate second-order no-gap optimality conditions
	for the optimal control problem \eqref{eq:optimal_control_example}.
	This is not possible with the techniques of
	\cite{CasasWachsmuthsBangBang}
	which only allow to produce sufficient optimality conditions.
\end{example}

\section{Conclusion}
\label{sec:conclusion}

Given the examples and applications in \Cref{sec:5,sec:advantages}, we may
conclude that the theoretical framework of \Cref{sec:second-order_conditions}
is indeed a handy tool in the derivation of (no-gap) second-order optimality
conditions for problems of the type \eqref{eq:problem}. Our results are not
only applicable in those situations where the classical assumptions of
polyhedricity and second-order regularity are satisfied (see
\Cref{thm:no-gap-SOC_poly,thm:no-gap-SOC_soreg}) but also allow to study
problems with more complicated admissible sets.
This is underlined by the no-gap second-order
conditions which were derived in the bang-bang setting of \Cref{subsec:bangbang}.
Further investigation is needed for the study of
the directional curvature functional in the presence of pointwise state
constraints, cf.\ the comments after \Cref{ex:polyhedric_not_osor}. We plan to
discuss this topic in a forthcoming paper.

	\bibliographystyle{siamplain}
	\bibliography{references}

\end{document}